\documentclass[a4paper,11pt]{amsart}
\usepackage{amsmath,inputenc,euscript,amssymb,geometry}
\usepackage{graphicx}
\usepackage{amssymb}
\usepackage{latexsym}
\usepackage{amssymb,amsbsy,amsmath,amsfonts,amssymb,amscd,color,ulem}
\usepackage{hyperref}

\newtheorem{lemma}{Lemma}[section]
\newtheorem{theorem}[lemma]{Theorem}
\newtheorem{prop}[lemma]{Proposition}

\newtheorem{remark}[lemma]{Remark}

\begin{document}
\title[Riemann's non-differentiable function and the binormal flow]{Riemann's non-differentiable function\\and the binormal curvature flow}
\author[V. Banica]{Valeria Banica}
    \address[V. Banica]{Sorbonne Universit\'e, CNRS, Universit\'e de Paris, Laboratoire Jacques-Louis Lions (LJLL), F-75005 Paris, France, and Institut Universitaire de France (IUF)} 
\email{Valeria.Banica@ljll.math.upmc.fr}
\author[L. Vega]{Luis Vega}
\address[L. Vega]{BCAM-UPV/EHU Bilbao, Spain, luis.vega@ehu.es} 
\email{lvega@bcamath.org}
\date\today
\maketitle

\begin{abstract}
We make a connection between a famous analytical object introduced in the 1860s by Riemann, as well as some variants of it, and a nonlinear geometric PDE, the binormal curvature flow. As a consequence this analytical object has a non-obvious nonlinear geometric interpretation. We recall that the binormal flow is a standard model for the evolution of vortex filaments. We prove the existence of solutions of the binormal flow with smooth trajectories that are as close as desired to curves with a multifractal behavior. Finally, we show that this behavior falls within the multifractal formalism of Frisch and Parisi, which is conjectured to govern turbulent fluids. \\
\end{abstract}
 
 \section{Introduction}
In this article we construct the graph of Riemann's non-differentiable function, and variants of it, by using the binormal curvature flow, a geometric flow of curves in three dimensions that is related to the evolution of vortex filaments. We also make a rigorous connection between the binormal flow and the multifractal formalism of Frisch and Parisi.

\subsection{Riemann's function and the multifractal formalism} A classical problem of mathematical analysis is finding real variable functions that are continuous but not differentiable at any point. Although it seems that the first example is due to Bolzano, traditionally the reference names in this matter are Riemann and Weierstrass, the latter attributing to Riemann the example
\begin{equation}
\label{Riemann}
\varphi_R(t)=\sum_{j=1}^\infty\frac{\sin(tj^2)}{j^2}.
\end{equation}
In fact Weierstrass, faced with the impossibility of proving that the previous function is not differentiable at any point, proposes his own examples which are later known as Weierstrass' functions. After fundamental contributions from Hardy in 1915 \cite{Har}, the problem was not solved until 1960 by Gerver, who proved in \cite{Ge70} and \cite{Ge71} that the function $\varphi_R$ is not differentiable except at times $t_{p,q}=\pi p/q$, with $p$ and $q$ odd numbers, in which case the derivative is precisely $-1/2$. Later Duistermaat \cite{Du} studied the self-similarity properties of the complex function, intimately associated with the Rieman's function $\varphi_{R}$, defined as:
 \begin{equation}
\label{Duis}
\varphi_D(t)=\sum_{j=1}^\infty\frac{e^{i tj^2}}{ij^2}.
\end{equation}
He drew attention to the apparent fractal properties of the graph generated by it. Finally, Jaffard proved using the wavelet transform in \cite{Ja} that in fact Riemann's function $\varphi_{R}$, and analogously the complex version studied by Duistermaat $\varphi_{D}$ is a multifractal function that moreover satisfies what is known as the multifractal formalism of Frisch and Parisi. The motivation of the latter notion has its roots in the theory developed by Frisch and Parisi to explain certain data obtained in \cite{AHGA} by Antonia, Hopfinger, Gagne and Anselmet  on the velocity structure functions in turbulent shear flows, that exit the homogeneous and isotropic framework of Kolmogorov 41's theory of turbulence.
 
More concretely, Jaffard's result in \cite{Ja} is about the spectrum of singularities, that is the function $d_f(\beta)$ which associates $\beta$ with the Hausdorff dimension of the sets of points $t_0$ where  $f$ has pointwise H\"older regularity of exponent $\beta$. This H\"older exponent is defined as the supremum of $\{\alpha: f\in \mathcal C^\alpha (t_0)\}$. Here $\mathcal C^\alpha (t_0)$ stands for the functions $f$ for which there exists a polynomial P of order at most $\alpha$ such that locally at $t_0$
 $$|f(t)-P(t-t_0)|<C|t-t_0|^\alpha.$$ 
 For instance Weierstrass' functions $W_{a,b}(x)=\sum_{n\in\mathbb N^*}a^n\cos(b^n x)$ with $a<1<ab$ are nowhere differentiable, but have constant H\"older exponent $\alpha=-\log a\slash \log b$. Thus they belong to the class of monofractal functions characterized by the fact that their spectrum support are reduced to one point, encoding, despite of the fractal appearance of its graph, a sort of disciplined irregularity. However, the points where the H\"older exponent is reached might be a fractal set, and actually this is the reason of the monofractal label. The devil's staircase is a famous example of monofractal function, as it has only one finite H\"older exponent, reached on the Cantor's triadic set.  In turn, multifractal functions are those whose H\"older exponent takes at least two finite values. The most complex such functions are those with spectrum positive at least on a whole interval. This encodes the fact that the regularity varies roughly between close points. For more details on these notions one can consult \cite{Ja2}.
 In \cite{Ja} it is proved that for $\beta \in [1/2,3/4]$,
\begin{equation}
\label{spectrum}d_{\varphi_R}(\beta)=4\beta-2.
\end{equation}
It was also shown  in \cite{Ja} that \eqref{spectrum} fits with what Frisch and Parisi conjecture in \cite{FP}:
\begin{equation}\label{FP}
d_{f}(\beta)=\inf_p (\beta p-\eta_{f}(p)+1),
\end{equation}
where $\eta_{f}(p)$ is defined in terms of Besov regularity:
\begin{equation}\label{Besov}
\eta_{f}(p)=\sup\{s,\,f\in B_p^{\frac sp,\infty}\}.
\end{equation}
We refer the reader to \S 8.5.3 of \cite{Fr} and p.443 of \cite{Ja} for the details on this multifractal formalism. Also, it was proved recently in \cite{BoEcVi} that Riemann's function is intermittent. The results in \cite{Ja} and \cite{BoEcVi} are analytical in nature, and no direct connection is established between Riemann's function and turbulence. The aim of this article is to make a connection between Riemann's function and the time evolution of vortex filaments. 
 
 \subsection{Vortex filaments: the binormal flow model and particular solutions}
 The vortex filaments are present in 3-D fluids having vorticity concentrated along a curve, and are a key element of quantum and classical fluid turbulent dynamics. This low regularity framework is difficult to analyze through Euler and Navier-Stokes equation. It is however at the heart of current investigations (see for instance \cite{JeSe},\cite{BGHG}). In this article we consider the binormal flow equation (BF), a classical reduced model for vortex filament dynamics. This model was formally derived by truncating the integral given by Biot-Savart's law (\cite{DaR},\cite{MuTaUkFu},\cite{ArHa},\cite{CaTi}). Recently a rigorous argument, but still under some strong assumptions, has been given by Jerrard and Seis in \cite{JeSe}. If the vorticity concentrates along a curve $\chi(t,x)$, where $t$ stands for the temporal variable and $x$ is the arclength parameter, the BF evolution is 
\begin{equation}\label{VFE}\partial_t\chi=\partial_x\chi\wedge\partial_x^2\chi.
\end{equation}
Using the Frenet system  is immediate to see that \eqref{VFE} is also written as
\begin{equation}
\label{bf}\partial_t\chi=\kappa  b,
\end{equation}
where $\kappa$ stands for the curvature of the curve and $ b$ for the binormal vector.
By differentiation with respect to $x$, the tangent vector $T$ of a BF solution solves the Schr\"odinger map with values in the unit sphere $\mathbb S^2$, that is the classical continuous Heisenberg model used in ferromagnetics
\begin{equation}
\label{SM}\partial_t T=T\wedge\partial_x^2T.
\end{equation}
Finally, and thanks to the Hasimoto transformation:
$$\psi(t,x)=\kappa(t,x) e^{i\int_0^x\tau(t,s)ds},$$
with $\tau$ denoting the torsion of the curve, one gets that the function $\psi(t)$, called filament function of $\chi(t)$, satisfies the 1-D focusing cubic Schr\"odinger equation 
\begin{equation}\label{cubic}\begin{array}{c}i\psi_t +\psi_{xx}+\frac 12\left(|\psi|^2-A(t)\right)\psi=0,
\end{array}
\end{equation}
for some real function $A(t)$ (\cite{Ha}). Conversely, from a solution of the 1-D cubic Schr\"odinger equation one can construct a solution of the binormal flow solving either the Frenet equations, or better through the construction of a parallel frame $(T,e_1,e_2)(t,x)$. This type of frame fits better with our needs because it is not necessary to suppose that $\kappa>0$ (for details on this construction see for instance \S 2 of \cite{BVARMA}). It is important to note that with this construction, the binormal flow solution obtained from $\psi(t,x)$ solution of \eqref{cubic} is the same as the one obtained from $\psi(t,x)e^{-i\frac{\Phi(t)}{2}}$ solution of \eqref{cubic} with $A(t)$ replaced by $A(t)-\Phi(t)$. Thus one can always reduce to the usual cubic nonlinear Schr\"odinger equation, i.e., \eqref{cubic} with $A(t)=0$.

Simple examples that can be obtained by this construction are:
\begin{itemize}
\item the straight line; $\psi(t,x)=0$ and $A(t)=0$;
\item the circle;  $\psi(t,x)=c>0$ and $A(t)=c^2$;
\item the helix;  $\psi(t,x)=c e^{ix\omega_0-it\omega_0^2}$ with $c>0$ and $A(t)=c^2$;
\item the self-similar solutions; $\psi(t,x)=\frac{c}{\sqrt t} e^{i\frac {x^2}{4t}}$ with $c>0$ and $A(t)=\frac{c^2}{t}$.
\end{itemize}
After integrating the frame system one gets  solutions of the non-linear equations \eqref{SM} and \eqref{VFE}. For doing that one needs to know the trajectory in time of one point. This is rather easy for the first three examples but is more delicate for the last one. As a matter of fact, it is better to solve directly \eqref{SM} and \eqref{VFE} to get the four examples mentioned above, instead of using \eqref{cubic}. 

For instance for the selfsimilar solutions  it is enough to look for solutions of the type $\chi(t,x)=\sqrt t G(x/\sqrt t)$. Then it is  easy to get that $G$ has to solve the non-linear ode
\begin{equation}
\label{G}\frac12G-\frac {x}{2}G'=G'\wedge G''.
\end{equation}
 From this is rather simple to conclude that $G$ is determined by the fact that the curvature has to be a constant $c$ and the torsion has to be $\tau(s)=s/2$, see \cite{Bu}. This means that $\chi(t)$ has curvature $\kappa(t,x)=\frac{c}{\sqrt{t}}$ and torsion $\tau(t,x)=\frac{x}{2t}$, and that $\chi(t)$ tends to two different lines at $x\pm \infty$ that are  the same at all times. As a consequence $\chi(t)$ is a smooth function for $t>0$ that becomes at $t=0$ a polygonal line with one corner located at $x=0$. 
 These selfsimiliar solutions were characterized in \cite{GRV}.  We will make a very strong use of this characterization in this paper. In particular the angle $\theta$ of the corner is related to $c$ by the formula
 \begin{equation}\label{angle}
\sin\frac{\theta} 2=e^{-\frac\pi 2c^2}.
\end{equation}
Recall that if $\chi_0$ is a polygonal line with just one corner of angle $\theta$ located at $x=0$, then its curvature is given by $\kappa(0,x)=(\pi-\theta) \delta(x)$. Nevertheless,  for constructing the solution of the binormal flow for that $\chi_0$,  one has to solve \eqref{cubic} with initial data $\psi(0,x)=c\delta(x)$ and $c$ as in \eqref{angle}. We will also need to know what is the relation  between $G(0)$ and the two asymptotic lines of  $G$ at infinity and the plane that contains them, see \cite{GRV}. Observe that from \eqref{G} we get that  the trajectory of the corner is
$$\chi(t,0)=\sqrt t G(0)=2\frac{c}{\sqrt t}b(0),$$
with $(T(0), n(0), b(0))$ the Frenet frame at $x=0$ of the profile curve $G$,  which can be taken any orthonormal matrix due to the rotation invariance of \eqref{G}. In this paper we will follow \cite{GRV} and take $(T(0), n(0), b(0))$ the canonical orthonormal basis of $\mathbb R^3$.

Similarly, the straight line, characterized by $\kappa=0$, is a trivial solution of \eqref{VFE}, and the circle and the helix can be easily obtained by looking at traveling solutions of \eqref{SM}. This immediately gives the dynamics of these solutions, and the particular fact that they conserve their shapes. Indeed, the circle moves with a constant speed along the axis perpendicular to the plane where it is contained, with direction depending on the initial orientation given by the arclength parametrization. The helix evolves by screwing up or down, also depending on the initial orientation. At this point it is important to recall that vortex filaments with the shapes of straight lines, circles, and helices  do exist, both in experiments and as solutions of Euler equations. Also the selfsimilar solutions are very reminiscent of the flow behind a delta wing jet and in the reconnection process of helium superfluid. We refer the reader to \cite{BV5} for the corresponding references. 

It is worth mentioning that the helix can be obtained from the circle using one of the symmetries of the set of solutions of \eqref{cubic}. These are the Galilean transformations: if $\psi(t,x)$ solves \eqref{cubic} with a constant $A(t)$, so does
$$\psi_{\omega_0}(t,x)=e^{ix\tau_0-it\omega_0^2}\psi(t,x-2\omega_0t)
$$ 
for any $\omega_0\in \mathbb R.$

\subsection{Numerical evidence about the connection between Riemann's function and the line, circle, and helix filaments} 
In \cite{DHV} the Galilean transformations are used to look for solutions of \eqref{VFE} that are initially a planar regular polygon of $M$ sides. The reason is simply because, in view of the construction of self-similar solutions, it is natural to look for solutions of the cubic Schr\"odinger equation with initial data
\begin{equation}
\label{Dirac}
\psi_M(0,x)=c_M  \sum_{j\in\mathbb Z} \delta (x-\frac{2\pi}{M}j)=c_M\frac{M}{2\pi}\sum_{j\in\mathbb Z} e^{ixMj},
\end{equation}
with $c_M>0$ related to the angle $\theta_M=\frac{M-2}{M}\pi$ by the relation \eqref{angle}, and $\delta$ denoting  Dirac's delta function. The last equality uses Poisson's summation formula 
\begin{equation}\label{P}
\sum_{j\in\mathbb Z}f(j)=\sum_{j\in\mathbb Z}\hat f(2\pi j)=\sum_{j\in\mathbb Z}\int e^{-i2\pi j y}f(y)dy.
\end{equation}
Then, it follows immediately from \eqref{Dirac} that the $\psi_M(0,x)$ is invariant under the discrete subgroup of the Galilean transformations given by $\omega_0\in\mathbb Z$. As a consequence, if uniqueness holds, it is proved in \cite{DHV} that then 
\begin{equation}
\label{Dirac2}
\psi_M(t,x)=\tilde c_M(t)  \sum_{j\in\mathbb Z}  e^{ixMj-it(Mj)^2},
\end{equation}
with $\tilde c_M(t)$ a real function which is determined by geometric means. Later on it was showed in \cite{BVCPDE} the existence of a formal conservation law whose validity implies that $\tilde c_M(t)$ should indeed be a constant so
$$\tilde c_M^2(t)=-\frac{M^2}{4\pi^2}\ln(\cos\frac{\pi}{M}).
$$
Notice that  for all $M$ we have 
$$\frac1 {\tilde c_M}\,\psi_M(\frac{t}{M^2},\frac{x}{M})=  \sum_{j=-\infty}^\infty e^{ixj-itj^2},
$$
and that $\lim_{M\to\infty}\tilde c_M=\frac 1{4\pi}$. Hence,
\begin{equation}
\label{Dirac3}
\lim_{M\to\infty}4\pi \,\psi_M(\frac{t}{M^2},\frac{x}{M})=  \sum_{j=-\infty}^\infty e^{ixj-itj^2},
\end{equation}
which is the  solution of the linear Schr\"odinger equation with periodic boundary conditions,
\begin{equation}\label{linear}\begin{array}{c}i\psi_t +\psi_{xx}=0
\\ \psi(0,x)=\sum_{j\in\mathbb Z}\delta(x-2\pi j).
\end{array}
\end{equation}
Moreover, this solution describes the Talbot effect in Optics, see \cite{Be}. In \cite{DHV} the consequences of this effect in \eqref{VFE} with initial data given by regular polygons were considered, suggesting a possible connection with the turbulent dynamics observed in non-circular jets (see on this subject for instance \cite{GrGuPa}). Let us mention also that at a less singular level, the Talbot effect for the linear and nonlinear  Schr\"odinger equations on the torus with initial data given by functions with bounded variation has been largely studied (\cite{Os},\cite{Ro},\cite{Ta},\cite{ET},\cite{ChErTz}). Let us mention also that the fractal behavior of one corner has been observed numerically in the context of the architecture of aortic valve fibers in \cite{PMQ} and \cite{SP}.

Immediately we obtain that fixing $x$ at the origin and integrating in time the limit in \eqref{Dirac3} we obtain\footnote{The term $j=0$ is understood to be $t$.},
\begin{equation}\label{1.0}
\int_0^t\sum_{j=-\infty}^\infty e^{-i\tau j^2}\,d\tau=\sum_{j\in\mathbb Z}\frac{e^{itj^2}-1}{ij^2}=:\frak{R}(t).
\end{equation}
In fact 
$$\frak{R}(t)=2\varphi_D(t)+t+\sum_{j\in\mathbb Z^*}\frac 1{j^2},$$ 
where $\varphi_D$ is the function defined in \eqref{Duis} and studied by Duistermaat. It follows from the same arguments given by Gerber that at the points where $\varphi_D(t)$ is differentiable the derivative is precisely $-1/2$. Therefore at those points $\frak{R'}(t)=0$. It turns out that when one looks at the trajectory in the complex plane of $\frak{R}(t)$, see Figure \ref{FigRiemann.pdf},
\begin{figure}[h]
\begin{center}
\includegraphics[scale=0.2]{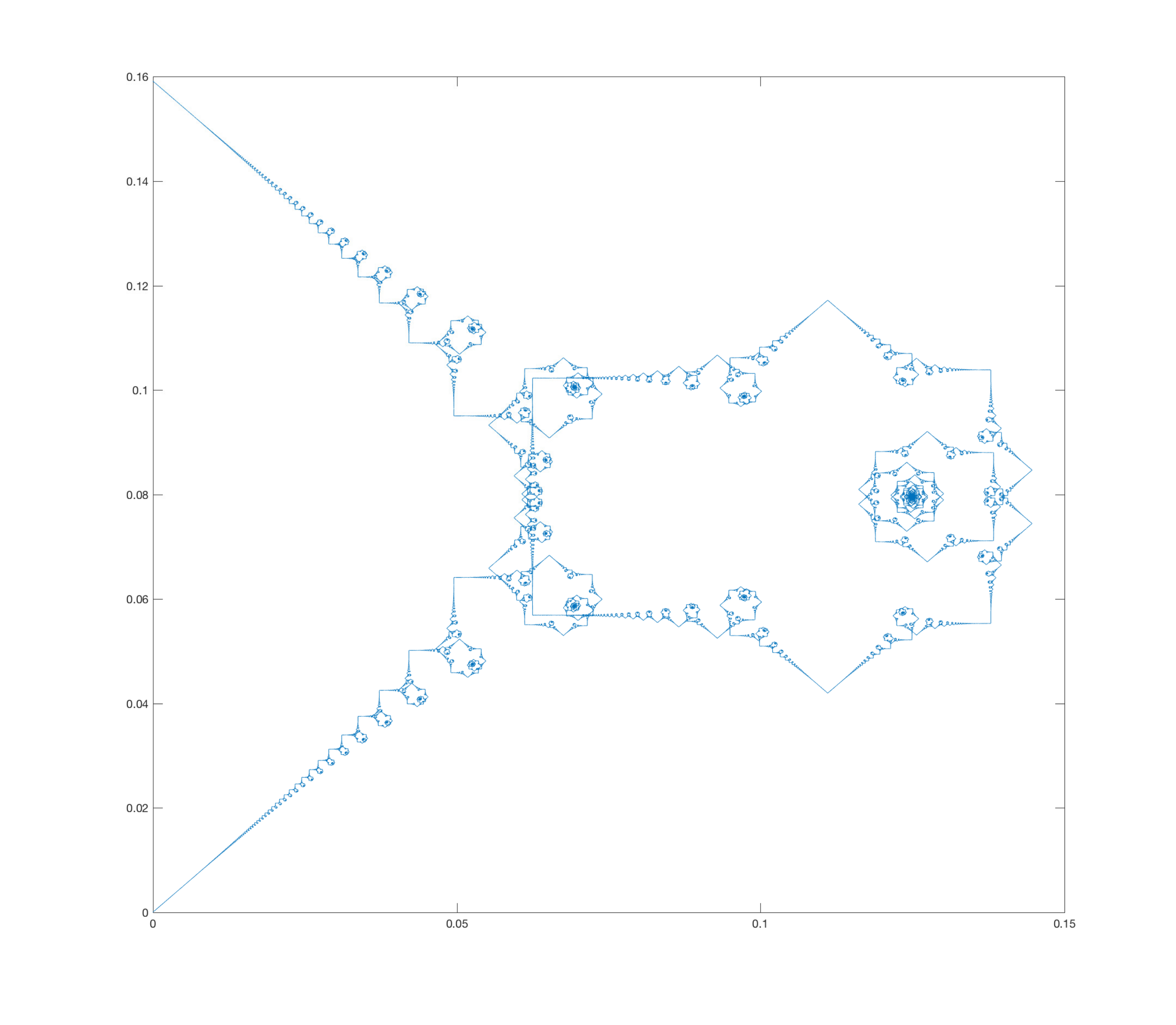}
\caption{\small {{Graph of $\frak{R}(t)=\int_0^t \sum_j e^{i\tau j^2}\,d\tau=\sum_{j\in\mathbb Z}\frac{e^{itj^2}-1}{ij^2}$}}}.
\label{FigRiemann.pdf}
\end{center}
\end{figure}
there is no tangent at those points due to the fact that the curve spirals around them. In fact, it has been recently proved by Eceizabarrena in \cite{Ec}  that  the trajectory, although continuous, and contrary to what happens with Riemann's function, does not have a tangent at any point.

Therefore, it was very natural to ask what is the trajectory in time of any of the corners of the M-regular polygon. In Figure \ref{Fig5} 
\begin{figure}[h]
\begin{center}
\includegraphics[width=4.5in]{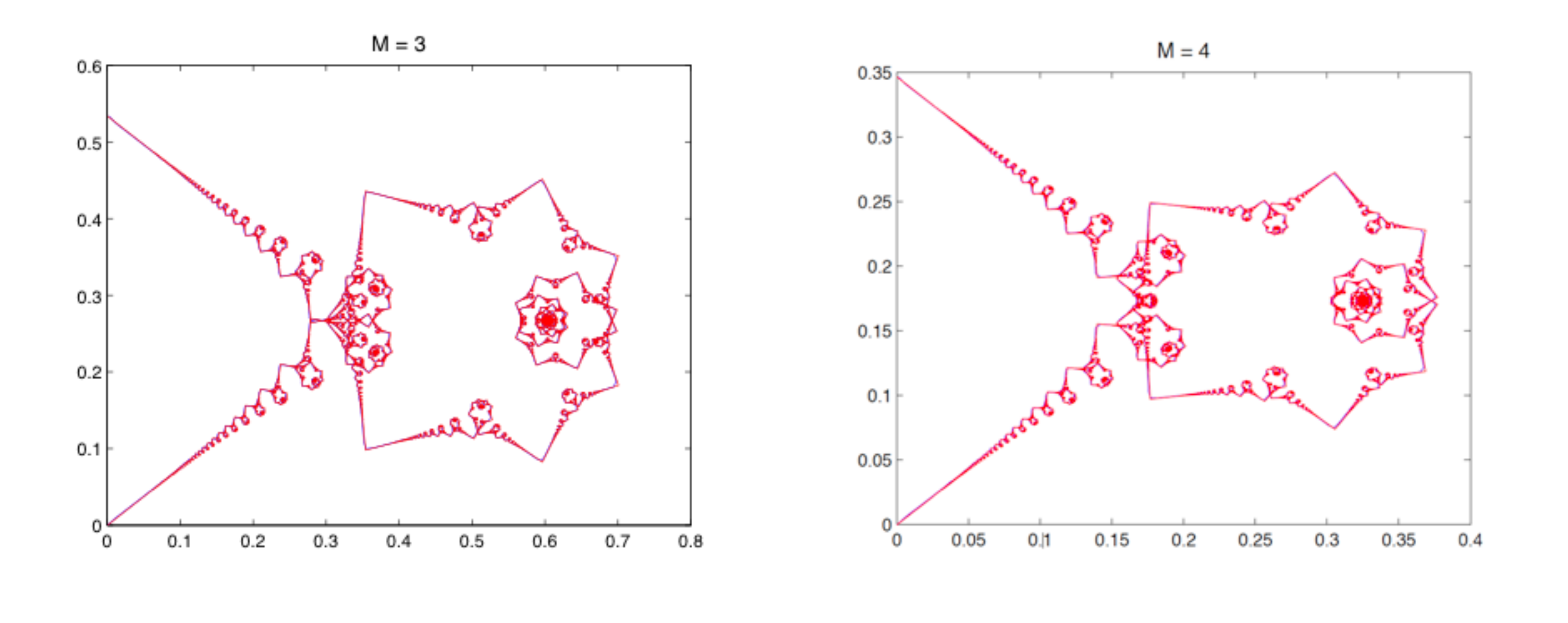}
\caption{\small {{Trajectory that at time $t=0$ starts in a vertex of an equilateral triangle (left) and of a square (right)}}}
\label{Fig5}
\end{center}
\end{figure}
we see the examples obtained in \cite{DHV} for $M=3$ and $M=4$. Similar pictures can be obtained for any $M$ and it becomes evident after looking at them, that they converge to the one of  $\frak{R}(t)$ given in Figure \ref{FigRiemann.pdf}. After doing an appropriate renormalization, this convergence is proved numerically in \cite{DHV}. Later on, it is also proved numerically in \cite{DHV2} the convergence of the Fourier coefficients of  the time derivative of the trajectory. These results have been extended in \cite{DHKV} to the case of helical regular polygons that converge to either a helix or to straight line. In the case of the helices, and depending on its pitch, different versions of $\frak{R}$ are obtained\footnote{The sequence of the squares that appears in  $\frak{R}$ has to be changed into the squares of any arithmetic progression with integer coefficients, see \cite{DHKV}.}.
The results in this paper prove analytically the aforementioned convergence using an approximation by non-closed polygonal lines.

\subsection{Presentation of the results}

Our main statement asserts the existence of various families of solutions $\{\chi_n\}_{n\in\mathbb N}$ of the binormal flow such that the trajectory of the corner $\chi_n(t,0)$ near $t=0$ is governed by the modified version of Riemann's function $\frak{R}$ as $n$ goes to infinity. 

\begin{theorem}\label{th} Let $n\in\mathbb N^*$, $\nu\in]0,1]$, $\Gamma>0$. There exist $T>0$, independent of $n$, and smooth solutions $\chi_n(t)$ of the binormal flow on $(-T,T)\setminus\{0\}$, weak solutions on $(-T,T)$, that at time $t=0$ become polygonal lines $\chi_n(0)$ with corners located at $j\in\mathbb Z$ with $|j|\leq n^\nu$, of same torsion $\omega_0\in\pi\mathbb Q$ and angles $\theta_n$ such that 
\begin{equation}\label{thetan}
\lim_{n\rightarrow\infty}n(\pi -\theta_n)=\Gamma,
\end{equation}
and
\begin{equation}\label{tang}
\chi_n(0,0)=(0,0,0),\quad \partial_x\chi_n(0,0^\pm)=(\sin\frac{\theta_{n}}2,\pm\cos\frac{\theta_{n}}2,0).\end{equation}

For these solutions we have the following description of the trajectory of the corner $\chi_n(t,0)$: \begin{equation}\label{cv}
n(\chi_{n}(t,0)-\chi_{n}(0,0))-(0,\Re(\frak{\tilde R}(t)),\Im(\frak{\tilde R}(t)))\overset{n\rightarrow\infty}\longrightarrow 0,
\end{equation}
uniformly on $(0,T)$. The function $\frak{\tilde R}$ is mutlifractal, its spectrum of singularities satisfies \eqref{spectrum} and the multifractal formalism formula \eqref{FP}. In the torsion-free case $\frak{\tilde R}(t)=-\Gamma\frac{\frak{R}(4\pi^2t)}{4\pi^2}$, with $\frak{R}$ given in \eqref{1.0}. The expressions for the cases with a non-trivial torsion are given in \eqref{Rtilde} and \eqref{varR1}.
\end{theorem}

In the torsion-free case the polygonal lines in Theorem \ref{th} can be chosen to approach the following special cases:
\begin{itemize}
\item the straight line; by taking $\nu<1$. Indeed, the total variation angle of $\chi_n(0,x)$ as $x$ varies from $-\infty$ to $\infty$ is 
$$\theta_{n}^{total}:=(\pi-\theta_n)(2\left \lfloor{n^\nu}\right \rfloor+1)\overset{n\rightarrow\infty}\approx \frac \Gamma n(2\left \lfloor{n^\nu}\right \rfloor+1),$$
so if $\nu<1$ we get convergence of $\theta_{n}^{total}$ to zero as $n$ goes to infinity.
\item a regular polygonal loop; by taking $\nu=1$ and $\theta_n=\frac {(2n-1)\pi}{2n+1}$. Indeed,   this means that the shape of $\chi_{n}(0)$ is composed of a regular closed polygon with $2n+1$ edges of size $1$, for  $|x|\leq n$, and of two half-lines as $|x|\geq n$.
\item a regular polygonal multi-loop; by increasing the number of corners of the regular polygonal loop to $x\in\{ j\in\mathbb Z, |j|\leq m n^\nu\}$, for $m\in\mathbb N^*$. The proof of the conclusion of Theorem \ref{th} goes the same, for  times $T_m$ of size ~$\frac 1m$. 
\end{itemize}
To approach other natural special cases we recall that the binormal flow is invariant under scaling: if $\chi$ is a solution then $\lambda\chi(\frac t{\lambda^2},\frac x\lambda)$ is a solution also for $\lambda>0$. Thus from Theorem \ref{th} we get for $\mu\in\mathbb R$ solutions of the binormal flow 
$$\tilde\chi_{n}(t,x)=\frac 1{n^\mu} \chi_{n}(n^{2\mu}t,n^\mu x).$$
For times smaller than $\frac T{n^{2\mu}}$ the convergence  \eqref{cv} becomes 
\begin{equation}\label{cvRloopcircle}
n^{1+\mu}(\tilde \chi_{n}(\frac t{n^{2\mu}},0)-\tilde\chi_{n}(0,0))-(0,\Re(\frak{\tilde R}(t)),\Im(\frak{\tilde R}(t)))\overset{n\rightarrow\infty}\longrightarrow 0,
\end{equation}
uniformly on $(0,T)$. This convergence is for instance valid for polygonal lines that tend to two lines at infinity and that locally approach the following curves::
\begin{itemize}
\item a circular loop; by rescaling the regular polygonal loop above with $\mu=1$. Indeed, $\tilde\chi_{n}(0)$ is composed by a regular closed polygon with $2n+1$ edges of size $\frac 1n$, thus inscribed in a circle of radius $\frac{1}{n\sin\frac\pi{2n+1}}$ for $|x|\leq n$ and two half-lines for $|x|\geq n$. In particular the polygon, as $n$ goes to infinity, converges to a circle of size $\frac 2\pi$. 
\item a circular multi-loop; by rescaling the regular polygonal multi-loop above with $\mu=1$. This example confirms the numerical simulations of \cite{DHKV} that can be seen in the video \url{https://www.youtube.com/watch?v=bwbpKvqGk-o&feature=youtu.be}.
\item the self-similar solution; by proceeding in the following way. Denote by $\theta$ the angle of a self-similar solution, and  choose $\theta_n=\pi-\frac\theta{2n+1}$ and $\nu=1$ so that $\theta_n^{total}=\theta$. Then, the shape of $\chi_{n}(0)$ is composed by a polygonal line with $2n+1$ corners with the same angle and edges of size $1$ that is inscribed in a circular sector of radius of size $2n+1$ for $|x|\leq n$, and of two half-lines for $|x|\geq n$. By rescaling  with $\mu>1$, we get that $\tilde\chi_n(0)$ is composed of a polygonal line with $2n+1$ corners with the same angle $\theta_n$, and edges of size $\frac 1{n^\mu}$ inside a circular sector of radius of size $\frac{1}{n^{\mu-1}}$ for $|x|\leq 1$, and of two half-lines for $|x|\geq 1$. Moreover the angle between the half-lines is precisely $\theta$. 
\end{itemize}

We note that in the above configurations the loops imply the existence of self-intersections, something that can not happen in a vortex filament. Nevertheless, they are relevant from a theoretical point of view as an analytical approximation to a real dynamics. Observe that the number of loops, although fixed, can be arbitrary large.

In the non-trivial torsion case the above examples give families of helicoidal polygonal lines. This way we have a non-planar approximation of the straight line, as well as, after rescaling, an approximation of a helical shape with as many turns as desired.\\

 In order to explain the proof of Theorem \ref{th} we have to recall some previous work done in \cite{BV5}  regarding the evolution through the binormal flow of non-closed polygonal lines $\chi_0(x)$ that tend as $x\rightarrow\pm\infty$ to two lines. These polygonal lines are characterized modulo a translation and a rotation by the fact that the corners are located at the integers $k\in \mathbb Z$ and by the curvature angles $\theta_k$ and torsion angles at the corners. For constructing the evolution of $\chi_0$ according to the binormal flow, we define first a sequence of complex numbers $\{\alpha_k\}$ in terms of the curvature and torsion angles of $\chi_0$. In particular, the identity \eqref{angle} has to be satisfied. The identity involving the torsion angles is more complicated, and is detailed in \S\ref{secthelix}.  We impose that for $s>1/2$ and $p=2$, the sequence $\{\alpha_k\}$ has to belong to   $l^{p,s}$,  the space of sequences of complex numbers 
 which is determined by the condition
$$\|\alpha_k\|_{l^{p,s}}^p:=\sum_{k\in\mathbb Z} |\alpha_k|^p(1+|k|)^{2s}<+\infty.$$ 
Then, we solve on $t>0$ the equation \eqref{cubic} with $A(t)=\frac{\sum_{k\in\mathbb Z}|\alpha_k|^2}{2\pi t}$ and 
with
 \begin{equation}\label{ansatz}\psi(t,x)=\sum_{k\in\mathbb Z}e^{-i\frac{|\alpha_k|^2}{4\pi}\log\sqrt{t}}(\alpha_k+R_k(t))e^{it\Delta}\delta(x-k),\end{equation}
 such that
\begin{equation}\label{controls}\sup_{0<t<T}\frac 1{t^{\gamma}}\|\{R_k(t)\}\|_{l^{2,s}}+t\|\{\partial_tR_k(t)\}\|_{l^{2,s}}<C.\end{equation}
Here $T$ and $C$ depend only on $\|\{\alpha_k\}\|_{l^{2,s}}$ and $\gamma\in(0,1)$. It is a remarkable fact that we have the mass conservation
\begin{equation}\label{mass}
\sum_{k\in\mathbb Z}|\alpha_k+R_k(t)|^2=\sum_{k\in\mathbb Z}|\alpha_k|^2,\quad\forall 0\,\leq t\leq T.
\end{equation}
Then from this solution $\psi$ we obtain in \cite{BV5} a smooth solution $\chi(t)$ of the binormal flow \eqref{VFE} on $(-T,T)\setminus\{0\}$, which is a weak solution on $(-T,T)$, that at time $t=0$ becomes the desired polygonal line $\chi_0$.


Thus to construct the solutions $\chi_n$ in Theorem \ref{th} we consider for $n\in\mathbb N^*$ sequences $\{\alpha_{n,k}\}_{k\in\mathbb Z}$ satisfying
\begin{equation}\label{defalpha}
\alpha_{n,k}=\left\{\begin{array}{c}c_ne^{ik\omega_0},\quad |k|\leq n^\nu,\\ 0, \quad |k|> n^\nu,\end{array}\right.
\end{equation}
where $c_n>0$ is defined by the identity \eqref{angle} in terms of the angle $\theta_n$ from the statement of Theorem \ref{th}. Note that we have
\begin{equation}\label{c_n}
c_n=\sqrt{-\frac 2\pi \arccos{\frac{\pi-\theta_n}2}}\overset{n\rightarrow\infty}{\approx}\frac \Gamma {2\sqrt \pi n}.
\end{equation}
Then the above results proved in \cite{BV5} insure us the existence of BF solutions that at time $t=0$ are polygonal lines with $2n^\delta+1$ corners located at  $x\in\{-\left \lfloor{n^\nu}\right \rfloor,...,-1,0,1,...,\left \lfloor{n^\nu}\right \rfloor\}$, all of them with the same curvature angle $\theta_n$ and torsion $\omega_0$. 
Having in mind that the binormal flow is invariant under translations and rotations we can consider  $\chi_n(0)$ such that \eqref{tang} holds. \\

Concerning the first part of Theorem \ref{th} we have to observe that if we directly use the results in \cite{BV5} we obtain solutions \eqref{ansatz} of the Schr\"odinger equation \eqref{cubic}, constructed by a Picard iteration procedure that is valid for times $T_n$ that depend on the weighted norm of $\|\{\alpha_{n,k}\}\|_{l^{2,s}}$, $s>1/2$. In particular, we obtain that $T_n$ vanishes as $n$ goes to infinity for $\delta=1$. In this article we shall improve the iteration procedure on $\{R_{n,k}\}_{k\in\mathbb Z}$  by using $l^1-$based spaces. This allows us to consider times  $T$ that depend only on $\|\alpha_{n,k}\|_{l^1_k}$, so that in our case $T$ can be chosen independent of $n$.  Another important remark is that the construction of the solution of the binormal flow \eqref{VFE} done in \cite{BV5}, and that is based on solutions of \eqref{cubic}, depends just on the $l^1$ norm of $\{\alpha_{n,k}\}_{k\in\mathbb Z}$ and of $\{R_{n,k}\}_{k\in\mathbb Z}.$ Therefore,
 the construction of the solutions of the binormal flow from the Schr\"odinger ones is assured by the results in \cite{BV5}.\par
Regarding the second part of the statement of the theorem, we will prove it by splitting the trajectory of the corner into three parts. One part will disappear in the limit due to an improved decay in $n$  that we obtained for  $R_{n,k}$ in $l^1$ and $l^{1,1}$.  Another part will be shown to be negligible by a fine analysis done in the construction of the parallel frame. It will be based on repeated integration by parts on some oscillatory integrals that naturally appear.  They include some problematic resonant terms that eventually disappear thanks to the specific value of $A(t)$. Finally, the last part, and in the torsion-free case, includes Riemann's function that appears thanks to Poisson summation formula. For helicoidal polygonal lines, the proof goes the same except that we end up with some variants of $\frak R$. In this case we prove, following the approach of Chamizo and Ubis in \cite{ChUb}, that their spectrum of singularities is the same as the one of Riemann's function in \eqref{spectrum}. Also, we show that the multifractal formalism of Frish and Parisi \eqref{FP} is satisfied, by using exponential sums estimates.\\

The paper is organized as follows. The proof of the first part of Theorem \ref{th} is done in \S \ref{sectionR}. The second part of Theorem \ref{th} in the torsion-free case is proved in \S \ref{sectionRiemm} using \S \ref{sectionR}, and two results related to the convergence of the normal vectors proved in \S \ref{sectiongn} and \S \ref{sectionNn} respectively. In the last section \S \ref{secthelix} we shall treat the nontrivial torsion cases of Theorem \ref{th}. This involves proving the results by Jaffard in \cite{Ja} in the more general setting of the squares of arithmetic progressions.\\

\section{Proof of the first part of Theorem \ref{th}}\label{sectionR}

In view of Proposition \ref{lemmaRk} below, for each sequence $\{\alpha_{n,k}\}_{k\in\mathbb Z}$ defined in \eqref{defalpha} we get a solution of the Schr\"odinger equation \eqref{cubic} with $A(t)=\frac{\sum_{k\in\mathbb Z}|\alpha_{n,k}|^2}{2\pi t}$, of type \eqref{ansatz}, with time of existence $T$ independent of $n$. As mentioned above in the Introduction, the construction of the corresponding BF solutions involves only the $l^1$ norms of $\{\alpha_{n,k}\}_{k\in\mathbb Z}$ and $\{R_{n,k}\}_{k\in\mathbb Z}$, so the first part of Theorem \ref{th} follows. \\

In the following Proposition we improve the fixed point argument on $\{R_{n,k}(t)\}_{k\in\mathbb Z}$ in a way it suits our purposes here.

\begin{prop}\label{lemmaRk} Let $n\in\mathbb N$, $\gamma\in(0,1)$, $q>1$, $C>0$ and $\{\alpha_{n,k}\}_{k\in\mathbb Z}$ a sequence such that $|\alpha_{n,k}|\leq \frac Cn$ for $|k|\leq n$ and $\alpha_{n,k}=0$ otherwise. 
There exist $T\in(0,1)$ depending only on $\gamma$ and $q$ and a unique solution written as
$$\sum_{k\in\mathbb Z}e^{-i\frac{|\alpha_{n,k}|^2}{4\pi}\log\sqrt{\tau}}(\alpha_{n,k}+R_{n,k}(t))e^{it\Delta}\delta_k(x),$$
of the equation
$$\begin{array}{c}i\psi_t +\psi_{xx}+\frac 12\left(|\psi|^2-\frac{\sum_k|\alpha_{n,k}|^2}{2\pi t}\right)\psi=0,
\end{array}$$
with the property:
\begin{equation}\label{Rkestn}\sup_{t\in(0,T)}\|t^{-\gamma}R_{n,k}(t)\|_{l^1_k}\leq C(\gamma,q) \frac1{n^{2-\frac 2q}},
\end{equation}
$$\sup_{\tau\in(0,T)}\|t\,\partial_t R_{n,k}(t)\|_{l^1_k}\leq C(\gamma,q) ,\quad \sup_{\tau\in(0,T)}\|t\,\partial_t R_{n,k}(t)\|_{l^q_k}\leq C(\gamma,q) \frac1{n^{1-\frac 1q}},$$
and
\begin{equation}\label{Rkestnw}\sup_{t\in(0,T)}\|t^{-\gamma}\,R_{n,k}(t)\|_{l^{1,1}_k}\leq C(\gamma,q) \frac1{n^{1-\frac 2q}}.
\end{equation}
\end{prop}
\begin{proof}
We follow the argument in \cite{BV5}, so that we have to find a fixed point for the application $\Phi(\{R_j\})=\{\Phi_k(\{R_j\})\}$ given by 
\begin{equation}\label{eqRk}\Phi_k(\{R_j\})(t)=-i\int_0^tf_k(\tau)d\tau+i\int_0^t\frac 1{8\pi \tau}(|\alpha_k+R_k(\tau)|^2-|\alpha_k|^2)(\alpha_k+R_k(\tau))d\tau,\end{equation}
where
\begin{equation}\label{nonrestpart}f_k(t)=\frac{1}{8\pi t}\sum_{(j_1,j_2,j_3)\in NR_k}e^{-i\frac{\Delta_{k,j_1,j_2,j_3}}{4\tau}}e^{-i\omega_{k,j_1,j_2,j_3}\log\sqrt{t}}(\alpha_{j_1}+R_{j_1}(t))\overline{(\alpha_{j_2}+R_{j_2}(t))}(\alpha_{j_3}+R_{j_3}(t)),\end{equation}
and
$$\omega_{k,j_1,j_2,j_3}=\frac{|\alpha_k|^2-|\alpha_{j_1}|^2+|\alpha_{j_2}|^2-|\alpha_{j_3}|^2}{4\pi},\quad \Delta_{k,j_1,j_2,j_3}=k^2-j_1^2+j_2^2-j_3^2,$$
$$NR_k=\{(j_1,j_2,j_3),j_1-j_2+j_3=k,\Delta_{k,j_1,j_2,j_3}\neq 0\},$$
see for instance (24) in \cite{BV5}; for simplicity we have ommited the $n-$subindex. We shall perform the fixed point argument in the ball
$$X^{\gamma,q,n}:=\{\{M_k\}\in\mathcal C^1((0,T),l^1)\cap \mathcal C((0,T),l^{1,1}),\quad \|\{M_k\}\|_{X^{\gamma,q,n}}<\delta\},$$
where
$$ \|\{M_k\}\|_{X^{\gamma,q,n}}:=n^{2-\frac 2q}\sup_{t\in(0,T)}\|t^{-\gamma}R_k(t)\|_{l^1}+\sup_{1\leq \tilde q\leq q} n^{1-\frac 1{\tilde q}}\sup_{t\in(0,T)}\|t\,\partial_t R_k(t)\|_{l^{\tilde q}}$$
$$+n^{1-\frac 2q}\sup_{t\in(0,T)}\|t^{-\gamma}\,R_{n,k}(t)\|_{l^{1,1}},$$
and $T\in(0,1)$ will be specified later.\\

Let $\{R_j\}\in X^{\gamma,q,n}$. 
We start with the estimates of $\|\Phi_k(\{R_j\})(t)\|_{l^1}$.  
To estimate the first term in the expression \eqref{eqRk} we shall perform as in (36) in \cite{BV5} an integration by parts in time to get advantage of the non-resonant phase $\Delta_{k,j_1,j_2,j_3}$ and to obtain integrability in time:
\begin{equation}\label{fk}i\int_0^t f_k(\tau)d\tau=t\sum_{(j_1,j_2,j_3)\in NR_k}\frac{e^{-i\frac{\Delta_{k,j_1,j_2,j_3}}{4\tau}}e^{-i\omega_{k,j_1,j_2,j_3}\log\sqrt{t}}}{2\pi \Delta_{k,j_1,j_2,j_3}}(\alpha_{j_1}+R_{j_1}(t))\overline{(\alpha_{j_2}+R_{j_2}(t))}(\alpha_{j_3}+R_{j_3}(t))
\end{equation}
$$-\int_0^t \sum_{(j_1,j_2,j_3)\in NR_k}\frac{e^{-i\frac{\Delta_{k,j_1,j_2,j_3}}{4\tau}}}{2\pi\Delta_{k,j_1,j_2,j_3}} \partial_\tau(e^{-i\omega_{k,j_1,j_2,j_3}\log\sqrt{t}}\tau (\alpha_{j_1}+R_{j_1}(\tau))\overline{(\alpha_{j_2}+R_{j_2}(\tau))}(\alpha_{j_3}+R_{j_3}(\tau)))\,d\tau.$$
We shall exploit the decay given by $\Delta_{k,j_1,j_2,j_3}=2(j_1-j_2)(j_3-j_2)$ on $NR_k$ 
yielding for $1\leq q<\infty$ the estimate
\begin{equation}\label{Young}
\left\|\sum_{(j_1,j_2,j_3)\in NR_k}\frac{M_{j_1}N_{j_2}P_{j_3}}{\Delta_{k,j_1,j_2,j_3}}\right\|_{l^1_k}\leq \sum_{j_2\notin\{j_1,j_3\}} \left|\frac{M_{j_1}N_{j_2}P_{j_3}}{(j_1-j_2)(j_3-j_2)}\right|
\leq C_q \|M_j\|_{l^q}\|N_j\|_{l^q}\|P_j\|_{l^1},
\end{equation}
obtained by performing H\"older estimates in the $j_1,j_2$ variables. Similarly, we get also as upper-bounds $C \|M_j\|_{l^q}\|N_j\|_{l^1}\|P_j\|_{l^q}$ and $C \|M_j\|_{l^1}\|N_j\|_{l^q}\|P_j\|_{l^q}$. Therefore
\begin{equation}\label{fkest}\|\int_0^t f_k(\tau)d\tau\|_{l^1}\leq Ct
(1+\|\alpha_j\|_{l^\infty}^2)
(\|\alpha_j\|_{l^q}^2\|\alpha_j\|_{l^1}+\sup_{\tau\in(0,T)}\|R_{j}(\tau))\|_{l^1}(\|\alpha_j\|_{l^1}^2+\sup_{\tau\in(0,T)}\|R_{j}(\tau))\|_{l^1}^2))\end{equation}
$$+Ct(\|\alpha_{j}\|_{l^q}^2+\sup_{\tau\in(0,T)}\|R_{j}(\tau))\|_{l^q}^2)\sup_{\tau\in(0,T)}\|\tau \partial_\tau R_{j}(\tau))\|_{l^1}.$$
The second term in \eqref{eqRk} contains only cubic terms with at least a power of $R_k$ so we conclude that for all $1\leq q<\infty$, as $l^q\subset l^1$,
\begin{equation}\label{Rkest}\|\Phi_k(\{R_j\})(t)\|_{l^1}\leq Ct
(1+\|\alpha_j\|_{l^\infty}^2)
(\|\alpha_j\|_{l^q}^2\|\alpha_j\|_{l^1}+\sup_{\tau\in(0,T)}\|R_{j}(\tau))\|_{l^1}(\|\alpha_j\|_{l^1}^2+\sup_{\tau\in(0,T)}\|R_{j}(\tau))\|_{l^1}^2))\end{equation}
$$+Ct(\|\alpha_{j}\|_{l^q}^2+\sup_{\tau\in(0,T)}\|R_{j}(\tau))\|_{l^1}^{2})\sup_{\tau\in(0,T)}\|\tau \partial_\tau R_{j}(\tau))\|_{l^1}$$
$$+C t^\gamma \sup_{\tau\in(0,T)}\|\tau^{-\gamma}R_{j}(\tau))\|_{l^1}(\|\alpha_j\|_{l^\infty}^2+\sup_{\tau\in(0,T)}\|R_{j}(\tau))\|_{l^\infty}^2).$$
In particular, as  $\|\alpha_j\|_{l^\infty}\leq \frac Cn, \|\alpha_j\|_{l^1}\leq \frac C{n^{1}}, \|\alpha_j\|_{l^q}\leq \frac{C}{n^{1-\frac 1 q}}$, we have
\begin{equation}\label{Rkest2}\sup_{\tau\in(0,T)}\|\tau^{-\gamma}\Phi_k(\{R_j\})(\tau)\|_{l^1}\leq CT^{1-\gamma}(\frac1{n^{2(1-\frac 1q)}}+\frac{T^{\gamma}\delta}{n^{2(1-\frac 1q)}}(1+\frac{T^{2\gamma}\delta^2}{n^{4(1-\frac 1q)}}))\end{equation}
$$+CT^{1-\gamma}(\frac1{n^{2(1-\frac 1q)}}+\frac{T^{2\gamma}\delta^{2}}{n^{2(1-\frac1q)}})\delta+C  \frac{\delta}{n^{2(1-\frac 1q)}}(\frac 1{n^2}+\frac{T^{2\gamma}\delta^2}{n^{4(1-\frac 1q)}}).$$
So for $T$ and $\delta$ less than a constant depending only on $\gamma$ and $q$ we have
$$n^{2(1-\frac 1q)}\sup_{\tau\in(0,T)}\|\tau^{-\gamma}\Phi_k(\{R_j\})(\tau)\|_{l^1}\leq \frac{\delta}{3}.$$

\bigskip 

Now we shall get estimates on $\partial_\tau \Phi_k(\{R_j\})(\tau)$. As we have for all $1\leq \tilde q$
$$\left\|\sum_{(j_1,j_2,j_3)\in NR_k}M_{j_1}N_{j_2}P_{j_3}\right\|_{l^{\tilde q}_k}\leq \|\{M_j\}\star \{N_j\}\star\{P_j\}(k)\|_{l^{\tilde q}_k}\leq C \|M_j\|_{l^1}\|N_j\|_{l^1}\|P_j\|_{l^{\tilde q}},$$
we get from \eqref{eqRk}
$$\sup_{\tau\in(0,T)}\|\tau\,\partial_\tau \Phi_k(\{R_j\})(\tau)\|_{l^{\tilde q}}\leq C \|\alpha_j+R_j(\tau)\|_{l^1}^2\|\alpha_j+R_j(\tau)\|_{l^{\tilde q}}$$
$$+C \|R_{j}(t))\|_{l^{\tilde q}}(\|\alpha_j\|_{l^\infty}^2+\sup_{\tau\in(0,T)}\|R_{j}(\tau)\|_{l^\infty}^2).$$
Hence for $1\leq \tilde q$, as $l^{\tilde q}\subset l^1$,
\begin{equation}\label{Rktest2}\sup_{\tau\in(0,T)}\|\tau\,\partial_\tau \Phi_k(\{R_j\})(\tau)\|_{l^{\tilde q}}\leq C(1+\frac{T^{2\gamma}\delta^{2}}{n^{4(1-\frac 1q)}}) (\frac1{n^{1-\frac 1{\tilde q}}}+\frac{T^\gamma\delta}{n^{2(1-\frac 1q)}})+C\frac{T^{\gamma}\delta}{n^{2(1-\frac 1q)}}(\frac 1{n^2}+\frac{T^{2\gamma}\delta^{2}}{n^{4(1-\frac 1q)}}).\end{equation}
Therefore, again for 
$T$ and $\delta$ less than a constant depending only on $\gamma$ and $q$ we have 
$$\sup_{1\leq\tilde q\leq q}n^{1-\frac 1{\tilde q}}\sup_{\tau\in(0,T)}\|\tau\partial_\tau \Phi_k(\{R_j\})(\tau)\|_{l^{\tilde q}}\leq \frac{\delta}{3}.$$

Finally, the control of the weighted norm $\|\{ \Phi_k(\{R_j\})(\tau)\|_{l^{1,1}}$ is obtained similarly, by using weighted estimates of type
$$\left\|\sum_{(j_1,j_2,j_3)\in NR_k}\frac{M_{j_1}N_{j_2}P_{j_3}}{\Delta_{k,j_1,j_2,j_3}}\right\|_{l^{1,1}}=\sum_k\left|\sum_{(j_1,j_2,j_3)\in NR_k}\frac{M_{j_1}N_{j_2}P_{j_3}}{\Delta_{k,j_1,j_2,j_3}}k\right|$$
$$\leq \sum_{j_2\notin\{j_1,j_3\}} \left|\frac{M_{j_1}N_{j_2}P_{j_3}(j_1-j_2+j_3)}{(j_1-j_2)(j_3-j_2)}\right|$$
$$\leq C \|M_j\|_{l^q}\|N_j\|_{l^q}\|P_j\|_{l^{1,1}}+C \|M_j\|_{l^1}\|N_j\|_{l^{\tilde q}}\|P_j\|_{l^{1}},$$
for all $1\leq\tilde q<\infty$. We get in the same way:
\begin{equation}\label{Rkest4}\sup_{\tau\in(0,T)}\|\tau^{-\gamma}\Phi_k(\{R_{j}\})(\tau)\|_{l^{1,1}}\leq CT^{1-\gamma}
(1+\|\alpha_j\|_{l^\infty}^2)
(\|\alpha_j\|_{l^q}^2\|\alpha_j\|_{l^{1,1}}+\|\alpha_j\|_{l^q}\|\alpha_j\|_{l^1}^2\end{equation}
$$+\sup_{\tau\in(0,T)}(\|R_{j}(\tau))\|_{l^{1,1}}+\|R_{j}(\tau))\|_{l^1})(\|\alpha_j\|_{l^1}^2+\sup_{\tau\in(0,T)}\|R_{j}(\tau))\|_{l^1}^2))$$
$$+CT^{1-\gamma}(\|\alpha_{j}\|_{l^q}\|\alpha_{j}\|_{l^{1,1}}\sup_{\tau\in(0,T)}\|\tau \partial_\tau R_{j}(\tau))\|_{l^q}+\|\alpha_{j}\|_{l^q}\|\alpha_{j}\|_{l^{1}}\sup_{\tau\in(0,T)}\|\tau \partial_\tau R_{j}(\tau))\|_{l^1})$$
$$+((\|\alpha_{j}\|_{l^{1,1}}+\|\alpha_j\|_{l^1})\sup_{\tau\in(0,T)}\|R_{j}(\tau))\|_{l^1}+\|\alpha_j\|_{l^1}\sup_{\tau\in(0,T)}\|R_{j}(\tau))\|_{l^{1,1}})\sup_{\tau\in(0,T)}\|\tau \partial_\tau R_{j}(\tau))\|_{l^1}$$
$$+(\sup_{\tau\in(0,T)}\|R_{j}(\tau))\|_{l^1}^2+\sup_{\tau\in(0,T)}\|R_{j}(\tau))\|_{l^1}\sup_{\tau\in(0,T)}\|R_{j}(\tau))\|_{l^{1,1}})\sup_{\tau\in(0,T)}\|\tau \partial_\tau R_{j}(\tau))\|_{l^1}$$
$$+C \sup_{\tau\in(0,T)}\|\tau^{-\gamma}R_{j}(\tau))\|_{l^{1,1}}(\|\alpha_j\|_{l^\infty}^2+\sup_{\tau\in(0,T)}\|R_{j}(\tau))\|_{l^\infty}^2).$$
Thus again for 
$T$ and $\delta$ less than a constant depending only on $\gamma$ and $q$ we have
$$n^{1-\frac 2q}\sup_{\tau\in(0,T)}\|\tau^{-\gamma}\Phi_k(\{R_{j}\})(\tau)\|_{l^{1,1}}\leq \frac{\delta}{3}.$$

Summarizing, we have obtained the existence of $T$ and $\delta$ less than a constant depending only on $\gamma$ and $q$ such that the stability estimate holds : if 
$\{R_k\}\in X^{\gamma,q,n}$ then $ \{\Phi_k(\{R_{j}\})\}\in X^{\gamma,q,n}.$
Thus to end the fixed point argument we need only the contraction estimates, that can be obtained in the same way.

\end{proof}

\section{Proof of the second part of Theorem \ref{th} in the planar case}\label{sectionRiemm}
Let us first recall from \S 4.3 in \cite{BV5} that for constructing the solutions of BF using the parallel frame $(T,e_1,e_2)$ the following equations have to be solved:
\begin{equation}
\label{Tx}
T_x=\Re u \,e_1+\Im u \,e_2=\Re(\overline u\, N),
\end{equation}

\begin{equation}
\label{Nx}N_x=e_{1x}+ie_{2x}=-\Re u\, T-i\Im u \,T=-u\, T,
\end{equation}

\begin{equation}
\label{Tt}T_t=-\Im u_x\, e_1+\Re u_x\,e_2=\Im (\overline{u_x}\,N),
\end{equation}

\begin{equation}
\label{Nt}N_t=-iu_x\, T+i\left(\frac{|u|^2}{2}-\frac{\sum_{k\in\mathbb Z}|\alpha_k|^2}{2t}\right) N,
\end{equation}

\begin{equation}
\label{chit}\chi_t=T\wedge T_x=T\wedge\Re(\overline{u}\,N)=\Im(\overline{u}\, N).
\end{equation}
Above we have taken
\begin{equation}\label{u}
u(t,x)=\sum_je^{-i(|\alpha_{j}|^2-\sum_{k\in\mathbb Z}|\alpha_k|^2)\log \sqrt{t}}(\alpha_{j}+R_{j}(t))\frac{e^{i\frac{(x-j)^2}{4t}}}{\sqrt{t}},
\end{equation}
$N=e_1+ie_2$, and we have omitted the $n-$subindices for simplicity. We note that the ansatz \eqref{u} comes from the one given in Proposition \ref{lemmaRk} applied to the sequence $\{\sqrt{4\pi i}\alpha_n\}$ instead of $\{\alpha_n\}$, and thus the notation $R_{j}$ in \eqref{u} comes from the remainder term in Proposition \ref{lemmaRk} divided by $\sqrt{4\pi i}$. Thus the remainder term in \eqref{u} enjoys the same decay properties \eqref{Rkestn}-\eqref{Rkestnw}.\par

In view of \eqref{chit} we can write the evolution of the corner located at $x=0$ as
$$\chi_n(t,0)-\chi_n(0,0)=\int_0^{t}\Im(\overline{u_n}N_n(\tau,0))\,d\tau$$
$$=\Im\int_{0}^{t}\sum_je^{i(|\alpha_{n,j}|^2-\sum_{k\in\mathbb Z}|\alpha_{n,k}|^2)\log \sqrt{\tau}}(\overline{\alpha_{n,j}+R_{n,j}(\tau)})\frac{e^{-i\frac{j^2}{4\tau}}}{\sqrt{\tau}}\,N_n(\tau,0)\,d\tau.$$
By using Proposition \ref{lemmaRk} the term involving $R_{n,j}(\tau)$ yields decay in $n$.
Therefore
$$\chi_n(t,0)-\chi_n(0,0)=\Im\int_{0}^{t}\sum_{|j|\leq n^\nu}\overline{\alpha_{n,j}}\frac{e^{-i\frac{j^2}{4\tau}}}{\sqrt{\tau}}\,e^{-i\sum_{k\neq j}|\alpha_{n,k}|^2\log \sqrt{\tau}}N_n(\tau,0)\,d\tau+r_n(t),
$$
with
\begin{equation}\label{rem}
|r_n(t)|\leq \frac{C}{n^{2^-}},\quad\forall t\in(0,T).
\end{equation}
Here $2^-$ means any number smaller than $2$, on which the constant depend. 

From Lemma 4.5 in \cite{BV5} we get the existence of the limit
$$\underset{t\rightarrow 0}{\lim}\,e^{i\sum_{j\neq x}|\alpha_{n,j}|^2\log\frac{|x-j|}{\sqrt t}}N_n(t,x)=:\tilde N_n(0,x)\in\mathbb S^2+i\mathbb S^2.$$
Hence we can write:
\begin{equation}\label{eq1}
\chi_n(t,0)-\chi_n(0,0)=\Im (e^{-ic_n^2\sum_{1\leq |j|\leq n^\nu}\log|j|} \tilde N_n(0,0)\int_{0}^{t}\sum_{|j|\leq n^\nu}\overline{\alpha_{n,j}}\frac{e^{-i\frac{j^2}{4\tau}}}{\sqrt{\tau}}\,d\tau)
\end{equation}
$$+\Im (e^{-ic_n^2\sum_{1\leq |j|\leq n^\nu}\log|j|}\int_{0}^{t}\sum_{|j|\leq n^\nu}\overline{\alpha_{n,j}}\frac{e^{-i\frac{j^2}{4\tau}}}{\sqrt{\tau}}g_n(\tau)\,d\tau)+r_n(t),$$
where
\begin{equation}\label{defgn}
g_{n}(t)=e^{i\Phi_n(t)}N_n(t,0)-\tilde N_n(0,0),\quad \Phi_n(t)=\sum_{j\in\mathbb Z}|\alpha_{n,j}|^2\log\frac{|j|}{\sqrt{t}}=c_n^2\sum_{1\leq |j|\leq n^\nu}\log\frac{|j|}{\sqrt{t}}.
\end{equation}

As we are in the case $\alpha_{n,j}=c_n$ for $|j|\leq n^\nu$, the first term makes appear the Riemann's function as follows.
\begin{lemma}\label{lemmaR}
$$|\int_{0}^{t}\sum_{|j|\leq n^\nu}\frac{e^{-i\frac{j^2}{4\tau}}}{\sqrt{\tau}}\,d\tau-\frac{e^{-i\frac\pi 4}}{2\pi\sqrt{\pi}}\,\frak{R}(4\pi^2 t)|\leq\frac{C}{n^\nu},$$
uniformly on $(0,T)$.
\end{lemma}
\begin{proof}
We first replace the summation in $j$ over the whole set of integers. Indeed, by integration by parts we have
$$|\int_0^t \sum_{|j|>n^\nu}\frac{e^{-i\frac{j^2}{4\tau}}}{\sqrt{\tau}}\,d\tau|=\left|\left[-\sum_{|j|>n^\nu}\frac{i4\tau\sqrt{\tau}e^{-i\frac{j^2}{4\tau}}}{j^2}\right]_0^t+\int_0^t \sum_{|j|>n^\nu}\frac{i6\sqrt{\tau}e^{-i\frac{j^2}{4\tau}}}{j^2}\,d\tau\right|\leq \frac{C}{n^\nu}.$$
Now we shall use Poisson's summation formula $\sum_{j\in\mathbb Z}f(j)=\sum_{j\in\mathbb Z}\hat f(2\pi j)$: 
$$\sum_{j\in\mathbb Z}e^{i 4\pi^2 tj^2}=\sum_{j\in\mathbb Z}\int e^{-i2\pi xj+i4\pi^2 tx^2}dx=\frac 1{\sqrt{4\pi^2 t}}\sum_{j\in\mathbb Z}\int e^{-iy\frac j{\sqrt{t}}+iy^2}dy$$
$$=\frac 1{2\pi \sqrt{ t}}\sum_{j\in\mathbb Z}\widehat{e^{i\cdot^2}}(\frac j{\sqrt{t}})=\frac {e^{i\frac \pi 4}}{2\sqrt{\pi}\sqrt{t}}\sum_{j\in\mathbb Z}e^{- i\frac{j^2}{4t}},$$
and the statement follows after integration in time.
\end{proof}
In view of this result, of \eqref{eq1} and \eqref{rem} we obtain, as $0<\nu\leq 1$,
$$n(\chi_n(t,0)-\chi_n(0,0))-\Im\,(nc_n\,e^{-ic_n^2\sum_{1\leq |j|\leq n^\nu}\log|j|}\tilde N_n(0,0)\frac{e^{-i\frac \pi 4}}{2\pi\sqrt{\pi}}\,\frak{R}(4\pi^2 t))$$
$$-\Im (nc_ne^{-ic_n^2\sum_{1\leq |j|\leq n^\nu}\log|j|}\int_{0}^{t}\sum_{|j|\leq n^\nu}\frac{e^{-i\frac{j^2}{4\tau}}}{\sqrt{\tau}}g_n(\tau)\,d\tau)\overset{n\rightarrow\infty}\longrightarrow 0,$$
uniformly on $(0,T)$. We note that if we use Lemma 4.5 of \cite{BV5}, the best estimate on $g_n$ we get involves one $l^{1,1}$ norm and a power of the $l^1$ norm of $\{\alpha_j\}$. In the present context this gives an undesired growth in $n$. Instead of that, we shall use Proposition \ref{propgn} to get
$$n(\chi_n(t,0)-\chi_n(0,0))-\Im\,(nc_n\,e^{-ic_n^2\sum_{1\leq |j|\leq n^\nu}\log|j|}\tilde N_n(0,0)\frac{e^{-i\frac \pi 4}}{2\pi\sqrt{\pi}}\,\frak{R}(4\pi^2 t))\overset{n\rightarrow\infty}\longrightarrow 0.$$
By using the convergence \eqref{c_n} of $c_n$ and the convergence of $\tilde N_n(0,0)$ obtained in Proposition \ref{cvN}:
$$\lim_{n\rightarrow\infty}\tilde N_n(0,0)=(0,\frac{1-i}{\sqrt{2}},\frac {-1-i}{\sqrt{2}}),$$
we have
$$n(\chi_n(t,0)-\chi_n(0,0))-\Im\,(\Gamma(0,-i,-1)\frac{\frak{R}(4\pi^2 t)}{4\pi^2})\,\overset{n\rightarrow\infty}\longrightarrow 0,$$
and Theorem \ref{th} follows.

\section{A convergence estimate for the normal vectors}\label{sectiongn}

\begin{prop}\label{propgn} Let $g_n$ be as defined in \eqref{defgn}. Then 
$$\int_{0}^{t}\sum_{|j|\leq n^\nu}\frac{e^{-i\frac{j^2}{4\tau}}}{\sqrt{\tau}}g_n(\tau)\,d\tau\overset{n\rightarrow\infty}\longrightarrow 0,$$
uniformly on $(0,T)$.
\end{prop}

\begin{proof}

In view of \eqref{Nt} and \eqref{u} we have, by omitting the subindices $n$ for simplicity, except for $c_n$,
$$\int_{0}^{t}\sum_{|j|\leq n^\nu}\frac{e^{-i\frac{j^2}{4\tau}}}{\sqrt{\tau}}g(\tau)\,d\tau=\int_{0}^{t}\sum_{|j|\leq n^\nu}\frac{e^{-i\frac{j^2}{4\tau}}}{\sqrt{\tau}}\int_{0}^{\tau}\left(-iu_x\, T+i\left(\frac{|u|^2}{2}-\frac{\sum_{j\in\mathbb Z}|\alpha_j|^2}{2s}\right) N+i\Phi_s N\right)e^{i\Phi}(s,0)dsd\tau$$
$$=-\int_{0}^{t} \sum_{|j|\leq n^\nu}\frac{e^{-i\frac{j^2}{4\tau}}}{\sqrt{\tau}} \int_0^\tau e^{i\left \lfloor{n^\nu}\right \rfloor c_n^2\log s}\sum_{k\neq 0} (\alpha_k+R_k(s))\frac{e^{i\frac{k^2}{4s}}}{2s\sqrt{s}}\,k\,T(s,0)e^{i\Phi(s)}dsd\tau$$
$$+i\int_{0}^{t} \sum_{|j|\leq n^\nu} \frac{e^{-i\frac{j^2}{4\tau}}}{\sqrt{\tau}}\int_0^\tau\sum_{k^2\neq l^2}(\alpha_k+R_k(s))(\overline{\alpha_l+R_l(s)})\frac{e^{i\frac{k^2-l^2}{4s}}}{2s}N(s,0)e^{i\Phi(s)}dsd\tau$$
$$+i\int_{0}^{t} \sum_{|j|\leq n^\nu} \frac{e^{-i\frac{j^2}{4\tau}}}{\sqrt{\tau}}\int_0^\tau\left(\sum_{k}(\alpha_k+R_k(s))(\overline{\alpha_{-k}+R_{-k}(s)})\frac{1}{2s}+\Phi_s\right)N(s,0)e^{i\Phi(s)}dsd\tau$$
\begin{equation}\label{IJK}=:I+J+K.\end{equation}

We start with the second term $J$. 
\begin{lemma}
$$J=i\int_{0}^{t} \sum_{|j|\leq n^\nu} \frac{e^{-i\frac{j^2}{4\tau}}}{\sqrt{\tau}}\int_0^\tau\sum_{k^2\neq l^2}(\alpha_k+R_k(s))(\overline{\alpha_l+R_l(s)})\frac{e^{i\frac{k^2-l^2}{4s}}}{2s}N(s,0)e^{i\Phi(s)}dsd\tau\overset{n\rightarrow\infty}{\longrightarrow}0.$$
\end{lemma}
\begin{proof}
Let us first observe that when $R_j(s)$ appears, it insures integration in $s$. Moreover ther is also enough decay in $n$ in view of \eqref{Rkestn}. On the remaining term
$$\tilde J:=i\int_{0}^{t} \sum_{|j|\leq n^\nu} \frac{e^{-i\frac{j^2}{4\tau}}}{\sqrt{\tau}}\int_0^\tau\sum_{k^2\neq l^2, |k|,|l|\leq n^\nu}\alpha_k\overline{\alpha_l}\frac{e^{i\frac{k^2-l^2}{4s}}}{2s}N(s,0)e^{i\Phi(s)}dsd\tau,$$
we shall perform an integration by parts
$$\tilde J=-\int_{0}^{t} \sum_{|j|\leq n^\nu} \frac{e^{-i\frac{j^2}{4\tau}}}{\sqrt{\tau}}\sum_{k^2\neq l^2, |k|,|l|\leq n^\nu}\alpha_k\overline{\alpha_l}\frac{e^{i\frac{k^2-l^2}{4\tau}}}{k^2-l^2}2\tau N(\tau,0)e^{i\Phi(\tau)}d\tau$$
$$+^2\int_{0}^{t} \sum_{|j|\leq n^\nu} \frac{e^{-i\frac{j^2}{4\tau}}}{\sqrt{\tau}}\int_0^\tau\sum_{k^2\neq l^2, |k|,|l|\leq n^\nu}\alpha_k\overline{\alpha_l}\frac{e^{i\frac{k^2-l^2}{4s}}}{k^2-l^2}(2sN(s,0)e^{i\Phi(s)})_sdsd\tau .$$
We also notice that
\begin{equation}\label{log}
|\sum_{k^2\neq l^2, |k|,|l|\leq n^\nu}\frac{1}{k^2-l^2}|\leq C(\log n^\nu)^2,
\end{equation}
and that in view of \eqref{Nt} and of the definition \eqref{defgn} of $\Phi$,
\begin{equation}\label{ders}
|(2sN(s,0)e^{i\Phi(s)})_s|
\leq C\frac{n}{\sqrt{s}}.
\end{equation}
Therefore we have got integrability in $s$ and $\tau$. Then, we have convergence to zero as $n$ goes to infinity for the boundary term and for the integral term with $j=0$. On the remaining integral terms with $j\neq 0$ we shall perform an integration by parts in $\tau$ to get summability in $j$ without loss in $n$:
$$\int_{0}^{t} \sum_{1\leq |j|\leq n^\nu} \frac{e^{-i\frac{j^2}{4\tau}}}{\sqrt{\tau}}\int_0^\tau\sum_{k^2\neq l^2, |k|,|l|\leq n^\nu}\alpha_k\overline{\alpha_l}\frac{e^{i\frac{k^2-l^2}{4s}}}{k^2-l^2}(2sN(s,0)e^{i\Phi(s)})_sdsd\tau$$
$$=\sum_{1\leq |j|\leq n^\nu} \frac{e^{-i\frac{j^2}{4t}}}{ij^2}4t\sqrt{t}\int_0^t\sum_{k^2\neq l^2, |k|,|l|\leq n^\nu}\alpha_k\overline{\alpha_l}\frac{e^{i\frac{k^2-l^2}{4s}}}{k^2-l^2}(2sN(s,0)e^{i\Phi(s)})_sds$$
$$-\int_{0}^{t} \sum_{1\leq |j|\leq n^\nu} \frac{e^{-i\frac{j^2}{4\tau}}}{ij^2}\left(4\tau\sqrt{\tau}\int_0^\tau\sum_{k^2\neq l^2, |k|,|l|\leq n^\nu}\alpha_k\overline{\alpha_l}\frac{e^{i\frac{k^2-l^2}{4s}}}{k^2-l^2}(2sN(s,0)e^{i\Phi(s)})_sds\right)_\tau d\tau.$$
Again by using \eqref{log} and \eqref{ders} we get the convergence to zero as $n$ goes to infinity.\\
\end{proof}

We consider now the first term and last terms in the decomposition \eqref{IJK}:
\begin{lemma}
$$I+K=-\int_{0}^{t} \sum_{|j|\leq n^\nu}\frac{e^{-i\frac{j^2}{4\tau}}}{\sqrt{\tau}} \int_0^\tau e^{i\left \lfloor{n^\nu}\right \rfloor c_n^2\log  s}\sum_{k\neq 0} (\alpha_k+R_k(s))\frac{e^{i\frac{k^2}{4s}}}{2s\sqrt{s}}\,k\,T(s,0)e^{i\Phi(s)}dsd\tau$$
$$+i\int_{0}^{t} \sum_{|j|\leq n^\nu} \frac{e^{-i\frac{j^2}{4\tau}}}{\sqrt{\tau}}\int_0^\tau\left(\sum_{k}(\alpha_k+R_k(s))(\overline{\alpha_{-k}+R_{-k}(s)})\frac{1}{2s}+\Phi_s\right)N(s,0)e^{i\Phi(s)}dsd\tau\overset{n\rightarrow\infty}{\longrightarrow}0.$$
\end{lemma}
\begin{proof}
The terms involving $\{R_k(s)\}$ in the first integral for $j=0$ and in the second integral converge all to zero as $n$ goes to $0$ by using \eqref{Rkestn} with $\gamma>\frac 12$. On the remaining terms, that involve $\{R_k(s)\}$ and $j\neq 0$ in the first integral, we integrate by parts in $\tau$ to get summability in $j$.  Eventually using again \eqref{Rkestn} with $\gamma>\frac 12$ we get convergence to zero as $n$ goes to $0$. 

Thus we have to show
\begin{equation}\label{tildeIK}
\tilde I+\tilde K=-\int_{0}^{t} \sum_{|j|\leq n^\nu}\frac{e^{-i\frac{j^2}{4\tau}}}{\sqrt{\tau}} \int_0^\tau e^{i\left \lfloor{n^\nu}\right \rfloor c_n^2\log s}\sum_{1\leq |k|\leq n^\nu} \alpha_k\frac{e^{i\frac{k^2}{4s}}}{2s\sqrt{s}}\,k\,T(s,0)e^{i\Phi(s)}dsd\tau
\end{equation}
$$+i\int_{0}^{t} \sum_{|j|\leq n^\nu} \frac{e^{-i\frac{j^2}{4\tau}}}{\sqrt{\tau}}\int_0^\tau\left(\sum_{ |k|\leq n^\nu}\alpha_k\overline{\alpha_{-k}}\frac{1}{2s}-\sum_{ |k|\leq n^\nu}|\alpha_k|^2\frac{1}{2s}\right)N(s,0)e^{i\Phi(s)}dsd\tau\overset{n\rightarrow\infty}{\longrightarrow}0.$$ 
In the case $\alpha_{n,k}=c_n$ for $|k|\leq n^\nu$ we have $\alpha_k=\alpha_{-k}$ and the term $\tilde K$ vanishes. Otherwise $\tilde K$ will cancel with a piece of $\tilde I$, as we shall see later. 
The term $\tilde I$ involves a bad power in $s$ for integration, so we need to integrate by parts in the $s$ variable:
$$\tilde I=-2i\int_{0}^{t} \sum_{|j|\leq n^\nu, 1\leq |k|\leq n^\nu}\alpha_k\frac{e^{i\frac{k^2-j^2}{4\tau}} }{k}\, \,e^{i\left \lfloor{n^\nu}\right \rfloor c_n^2\log{\tau}}T(\tau,0)e^{i\Phi(\tau)}d\tau$$
$$+2i\int_{0}^{t} \sum_{|j|\leq n^\nu}\frac{e^{-i\frac{j^2}{4\tau}}}{\sqrt{\tau}} \int_0^\tau\sum_{1\leq |k|\leq n^\nu}\alpha_k \frac{e^{i\frac{k^2}{4s}}}{k}\,(e^{i\left \lfloor{n^\nu}\right \rfloor c_n^2\log{s}}\sqrt{s}\,T(s,0)e^{i\Phi(s)})_sdsd\tau.$$

We first treat the boundary term. When $j^2=k^2$ we get a $\frac{\log n}{n}$ decay. When $j^2\neq k^2$ we perform an integration by parts and get decay in $n$ for all terms except the one coming from $T_\tau$ by using the estimate
\begin{equation}\label{sumjk}
\alpha_n\sum_{|j|\leq n^\nu, 1\leq |k|\leq n^\nu, j^2\neq k^2}\left|\frac{1}{k(k^2-j^2)}\right|\leq \frac{C}{n^{1^-}}.
\end{equation}
We are left with the term involving $T_\tau$:
$$I_b:=-2i\int_{0}^{t} \sum_{||j|\leq n^\nu, 1\leq |k|\leq n^\nu, j^2\neq k^2}\alpha_k\frac{e^{i\frac{k^2-j^2}{4\tau}} }{k(k^2-j^2)}\, \,e^{i\left \lfloor{n^\nu}\right \rfloor c_n^2\log{\tau}}\Im(\overline{u_x}N(\tau,0))e^{i\Phi(\tau)}\tau^2d\tau$$
$$=-i\int_{0}^{t} \sum_{|j|\leq n^\nu, 1\leq |k|\leq n^\nu, j^2\neq k^2}\alpha_k\frac{e^{i\frac{k^2-j^2}{4\tau}}}{k(k^2-j^2)}\,  \,e^{i\left \lfloor{n^\nu}\right \rfloor c_n^2\log{\tau}}$$
$$\times\Im (i\sum_le^{-i\left \lfloor{n^\nu}\right \rfloor c_n^2\log{\tau}}(\overline{\alpha_l+R_l(\tau)})le^{-i\frac{l^2}{4\tau}}N(\tau,0))e^{i\Phi(\tau)}\sqrt{\tau}d\tau.$$
In view of \eqref{sumjk} and \eqref{Rkestnw} we obtain the convergence to zero as $n$ goes to infinity for the term involving $R_l$.  So it remains to estimate the integral
$$\tilde I_b:=-i\int_{0}^{t} \sum_{|j|\leq n, 1\leq |k|\leq n, j^2\neq k^2}\alpha_k\frac{e^{i\frac{k^2-j^2}{4\tau}} }{k(k^2-j^2)}\, \,e^{i\left \lfloor{n^\nu}\right \rfloor c_n^2\log{\tau}}$$
$$\times\Im (i\sum_{1\leq |l|\leq n^\nu}e^{-i\left \lfloor{n^\nu}\right \rfloor c_n^2\log{\tau}}\overline{\alpha_l}\,l\,e^{-i\frac{l^2}{4\tau}}N(\tau,0))e^{i\Phi(\tau)}\sqrt{\tau}d\tau.$$
For $l^2=\pm(k^2-j^2)$ the summation in $l$ disappears and we get again by \eqref{sumjk} a $\frac{\log^2 n}{n}$ decay. For $l^2\neq \pm(k^2-j^2)$ we perform an integration by parts and for the worse term involving the $u_xT$ in $N_\tau$ we use
\begin{equation}\label{sumjk2}
c_n^2\sum_{|j|\leq n^\nu, 1\leq |k|,|l|\leq n^\nu, j^2\neq k^2, l^2\neq \pm(k^2-j^2), p}\frac{|(\alpha_p+R_p(\tau))\,p|}{|k(k^2-j^2)(k^2-j^2\pm l^2)|}\leq \frac{C}{n^{1^-}},
\end{equation}
having in mind \eqref{Rkestnw} for the part concerning the $R_p$ terms.\\

Finally we treat the last remaining term, namely the integral term in $\tilde I$:
$$2i\int_{0}^{t} \sum_{|j|\leq n^\nu}\frac{e^{-i\frac{j^2}{4\tau}}}{\sqrt{\tau}} \int_0^\tau \sum_{1\leq |k|\leq n^\nu}\alpha_k \frac{e^{i\frac{k^2}{4s}}}{k}\,(e^{i\left \lfloor{n^\nu}\right \rfloor c_n^2\log{s}}\sqrt{s}\,T(s,0)e^{i\Phi(s)})_sdsd\tau$$
$$=2i\int_{0}^{t} \sum_{|j|\leq n^\nu}\frac{e^{-i\frac{j^2}{4\tau}}}{\sqrt{\tau}} \int_0^\tau\sum_{1\leq |k|\leq n^\nu}\alpha_k \frac{e^{i\frac{k^2}{4s}}}{k}\,(e^{i\left \lfloor{n^\nu}\right \rfloor c_n^2\log{s}}\sqrt{s}\,e^{i\Phi(s)})_sT(s,0)e^{i\Phi(s)}dsd\tau$$
$$+2i\int_{0}^{t} \sum_{|j|\leq n^\nu}\frac{e^{-i\frac{j^2}{4\tau}}}{\sqrt{\tau}} \int_0^\tau \sum_{1\leq |k|\leq n^\nu} \alpha_k\frac{e^{i\frac{k^2}{4s}}}{k}\,e^{i\left \lfloor{n^\nu}\right \rfloor c_n^2\log{s}}\sqrt{s}\,T_s(s,0)\,e^{i\Phi(s)}dsd\tau.$$
For the first integral we observe that
$$|(e^{i\left \lfloor{n^\nu}\right \rfloor c_n^2\log{s}}\sqrt{s}\,e^{i\Phi(s)})_s|\leq C(\frac{1}{n\sqrt{s}}+\frac{1}{\sqrt{s}}+\frac1{ns}).$$
Then, if $j=0$, we get $\frac{\log n}{n}$ decay. For the terms $j\neq 0$ we integrate by parts in $\tau$ to get summability in $j$, so that we get a $\frac{\log n}{n}$ bound. Therefore we are left with the second integral
$$2i\int_{0}^{t} \sum_{|j|\leq n^\nu}\frac{e^{-i\frac{j^2}{4\tau}}}{\sqrt{\tau}} \int_0^\tau \sum_{1\leq |k|\leq n^\nu} \alpha_k\frac{e^{i\frac{k^2}{4s}}}{k}\,e^{i2\left \lfloor{n^\nu}\right \rfloor c_n^2\log\sqrt{s}}\sqrt{s}\,\Im(\overline{u_x}N(s,0))\,e^{i\Phi(s)}dsd\tau$$
$$=i\int_{0}^{t} \sum_{|j|\leq n^\nu}\frac{e^{-i\frac{j^2}{4\tau}}}{\sqrt{\tau}} \int_0^\tau \sum_{1\leq |k|\leq n}\alpha_k \frac{e^{i\frac{k^2}{4s}}}{k}\,e^{i2\left \lfloor{n^\nu}\right \rfloor c_n^2\log\sqrt{s}}$$
$$\times \Im(i\sum_{l\neq 0}e^{-i2n\left \lfloor{n^\nu}\right \rfloor c_n^2\log\sqrt{s}}(\overline{\alpha_l+R_l(s)})\frac{l}{s}e^{-i\frac{l^2}{4s}}N(s,0))\,e^{i\Phi(s)}dsd\tau.$$
For the $R_l$ terms we use \eqref{Rkestnw} to get integrability in $s$ and a $\frac{1}{n^{1^-}}$ decay. 
We are left with
$$I_i:=i\int_{0}^{t} \sum_{|j|\leq n^\nu}\frac{e^{-i\frac{j^2}{4\tau}}}{\sqrt{\tau}} \int_0^\tau \sum_{1\leq |k|\leq n^\nu}\alpha_k \frac{e^{i\frac{k^2}{4s}}}{k}\,e^{i\left \lfloor{n^\nu}\right \rfloor c_n^2\log{s}}$$
$$\times \Im(i\sum_{1\leq |l|\leq n^\nu}e^{-i\left \lfloor{n^\nu}\right \rfloor c_n^2\log{s}}\,\,\overline{\alpha_l}\frac{\,l}{s}\,e^{-i\frac{l^2}{4s}}N(s,0))\,e^{i\Phi(s)}dsd\tau$$
$$=\frac i2\int_{0}^{t} \sum_{|j|\leq n^\nu}\frac{e^{-i\frac{j^2}{4\tau}}}{\sqrt{\tau}} \int_0^\tau \sum_{1\leq |k|,|l|\leq n^\nu}\alpha_k\overline{\alpha_l} \frac{l}{k}e^{i\frac{k^2-l^2}{4s}}\frac{1}{s}\,N(s,0)\,e^{i\Phi(s)}dsd\tau$$
$$+\frac i2\int_{0}^{t} \sum_{|j|\leq n^\nu}\frac{e^{-i\frac{j^2}{4\tau}}}{\sqrt{\tau}} \int_0^\tau\sum_{1\leq |k|,|l|\leq n^\nu} e^{i\left \lfloor{n^\nu}\right \rfloor c_n^2\log{s}}\,\alpha_k\alpha_l\frac{l}{k}e^{i\frac{k^2+l^2}{4s}}\frac{1}{s}\,\overline{N(s,0)}\,e^{i\Phi(s)}dsd\tau.$$
In the case $\alpha_{n,k}=c_n$ for $|k|\leq n^\nu$, in the first integral the terms with $k=l$ and the terms with $k=-l$ cancel. Otherwise they cancel with $\tilde K$ in \eqref{tildeIK}. For all the remaining terms we need to perform an integration by parts to settle the integration in $s$. The terms not involving the $u_xT$ term of $N_s$ are converging to zero as $n$ goes to infinity since
\begin{equation}\label{sumjk2}
c_n^2\sum_{1\leq |k|,|l|\leq n^\nu, k^2\neq l^2}\frac{|l|}{|k(k^2\pm l^2)|}\leq \frac{C}{n^{2^-}}.
\end{equation}
It goes the same by using \eqref{Rkestnw} for the $R_p$-part coming from the $u_xT$ term of $N_s$. Therefore the last terms to treat are
$$\int_{0}^{t} \sum_{|j|\leq n^\nu}\frac{e^{-i\frac{j^2}{4\tau}}}{\sqrt{\tau}} \int_0^\tau\sum_{1\leq |k|,|l|,|p|\leq n^\nu}\alpha_k\overline{\alpha_l} \alpha_p\frac{l\,p}{k(k^2-l^2)}e^{i\frac{k^2-l^2+p^2}{4s}}\frac{e^{i\left \lfloor{n^\nu}\right \rfloor c_n^2\log{s}}}{\sqrt{s}}\,T(s,0)\,e^{i\Phi(s)}dsd\tau,$$
and
$$\int_{0}^{t} \sum_{|j|\leq n^\nu}\frac{e^{-i\frac{j^2}{4\tau}}}{\sqrt{\tau}} \int_0^\tau\sum_{1\leq |k|,|l|,|p|\leq n^\nu} \alpha_k\alpha_l\overline{\alpha_p}\frac{\,l\,p}{k(k^2+l^2)}e^{i\frac{k^2+l^2-p^2}{4s}}\frac{e^{i\left \lfloor{n^\nu}\right \rfloor c_n^2\log{s}}}{\sqrt{s}}\,T(s,0)\,e^{i\Phi(s)}dsd\tau.$$
For $j=0$ we get again by \eqref{sumjk2} a $\frac{1}{n^{1^-}}$ decay. For $j\neq 0$ we perform a last integration by parts in $\tau$ to get summability in $j$ without loss in $n$ and use again \eqref{sumjk2} to get a $\frac{1}{n^{1^-}}$ decay.
\end{proof}
In view of \eqref{IJK} and the last two lemmas the Proposition \ref{propgn} is proved. 

\end{proof}

\bigskip

\section{Convergence of the modulated normal vectors at $(t,x)=(0,0)$}\label{sectionNn}
\begin{prop}\label{cvN}
The following convergence holds:
$$\lim_{n\rightarrow\infty}\tilde N_n(0,0)=(0,\frac{1-i}{\sqrt{2}},\frac {-1-i}{\sqrt{2}}).$$
\end{prop}

\begin{proof} Recall that from Lemma 4.5 in \cite{BV5} we have 
$$\underset{t\rightarrow 0}{\lim}\,e^{i\sum_{j\neq x}|\alpha_{n,j}|^2\log\frac{|x-j|}{\sqrt t}}N_n(t,x)=:\tilde N_n(0,x),$$
with 
$$|e^{i\sum_{j}|\alpha_{n,j}|^2\log\frac{|x-j|}{\sqrt t}}N_n(t,x)-\tilde N_n(0,x)|\leq  nC\frac{\sqrt{t}}{x},\quad x\in(0,\frac 14),$$
$$|e^{i\sum_{j\neq 0}|\alpha_{n,j}|^2\log\frac{|j|}{\sqrt t}}N_n(t,0)-\tilde N_n(0,0)|\leq  nC\sqrt{t}.$$
The growth in $n$ comes from $\|\{\alpha_j\}\|_{l^{1,1}}$. Moreover, Lemma 4.6 in \cite{BV5} insures us that 
$$\tilde N_n(0,x_1)=\tilde N_n(0,x_2), \quad x_1,x_2\in(0,\frac 14).$$
We have for $t\in(0,T)$ and  $x\in(0,\frac 14)$:
$$|\tilde N_n(0,0)-\tilde N_n(0,0^+)|\leq |\tilde N_n(0,0)-e^{i\sum_{j\neq 0}|\alpha_{n,j}|^2\log\frac{|j|}{\sqrt t}}N_n(t,0)|+$$
$$|N_n(t,0)-N_n(t,x)|+|e^{i\sum_{j\neq 0}|\alpha_{n,j}|^2\log\frac{|j|}{|x-j|}}-1||N_n(t,x)|$$
$$+|e^{i|\alpha_{n,0}|^2\log \frac{x}{\sqrt{t}}}-1||N_n(t,x)|+|e^{i\sum_{j}|\alpha_{n,j}|^2\log\frac{|x-j|}{\sqrt t}}N_n(t,x)-\tilde N_n(0,x)|.$$
In view of the above estimates we get for $t\in(0,T)$ and $x=\frac 18$:
\begin{equation}\label{splitN}
|\tilde N_n(0,0)-\tilde N_n(0,0^+)|\leq nC\sqrt{t}+C\frac{\log n}{n^2}+C\frac{\log t}{n^2}+|N_n(t,0)-N_n(t,x)|.
\end{equation}

Now for the last term, we use \eqref{Nx} to write
$$N_n(t,0)-N_n(t,x)=\int_0^x(-uT)(t,y)dy=-\int_0^x\sum_je^{i\left \lfloor{n^\nu}\right \rfloor c_n^2\log {t}}(\alpha_{n,j}+R_{n,j}(t))\frac{e^{-i\frac{(y-j)^2}{4t}}}{\sqrt{t}}\,T_n(t,y)dy.$$
By using Proposition \ref{lemmaRk} we obtain that the term involving $R_{n,j}$ is controlled by $\frac{C}{n^{2^-}}$. On the remaining part we perform an integration by parts for $j\neq 0$:
$$\frac{e^{i\left \lfloor{n^\nu}\right \rfloor c_n^2\log {t}}}{\sqrt{t}}\int_0^x\sum_{1\leq|j|\leq n^\nu}\alpha_je^{-i\frac{(y-j)^2}{4t}}\,T_n(t,y)dy=\left[e^{i\left \lfloor{n^\nu}\right \rfloor c_n^2\log {t}}2i\sqrt{t}\sum_{1\leq|j|\leq n^\nu}\alpha_j\frac{e^{-i\frac{(y-j)^2}{4t}}}{y-j}\,T_n(t,y)\right]_0^x$$
$$-e^{i\left \lfloor{n^\nu}\right \rfloor c_n^2\log {t}}2i\sqrt{t}\int_0^x\sum_{1\leq|j|\leq n^\nu}\alpha_je^{-i\frac{(y-j)^2}{4t}}(\frac{T_n(t,y)}{y-j})_y\,dy.$$
As $\partial_yT_n$ is upper-bounded by $\frac C{\sqrt{t}}$ we get that all this part is controlled by $C\frac {\log n}{n}$. We are left with the term for $j=0$:
$$\alpha_0\frac{e^{i\left \lfloor{n^\nu}\right \rfloor c_n^2\log {t}}}{\sqrt{t}}\int_0^x e^{-i\frac{y^2}{4t}}\,T_n(t,y)dy.$$
The integral on $y\in(0,\sqrt{t})$ is upper-bounded by $\frac{C}{n}$. On the remaining region of integration we perform an integration by parts
$$\alpha_0\frac{e^{i\left \lfloor{n^\nu}\right \rfloor c_n^2\log {t}}}{\sqrt{t}}\int_{\sqrt{t}}^x e^{-i\frac{y^2}{4t}}\,T_n(t,y)dy=\left[\alpha_0e^{i\left \lfloor{n^\nu}\right \rfloor c_n^2\log {t}}2i\sqrt{t} \frac{e^{-i\frac{y^2}{4t}}}{y}\,T_n(t,y)\right]_{\sqrt{t}}^x$$
$$-\alpha_0e^{i\left \lfloor{n^\nu}\right \rfloor c_n^2\log {t}}2i\sqrt{t} \int_{\sqrt{t}}^xe^{-i\frac{y^2}{4t}}(\frac{T_n(t,y)}{y})_y\,dy.$$
Again as $\partial_yT_n$ is upper-bounded by $\frac C{\sqrt{t}}$ we get that this part is controlled by $C\frac {\log t}{n}$.

Summarizing we have obtained that for $t\in(0,T)$
\begin{equation}\label{splitNbis}
|\tilde N_n(0,0)-\tilde N_n(0,0^+)|\leq nC\sqrt{t}+C\frac{\log n}{n}+C\frac {\log t}{n}.
\end{equation}
By choosing $t$ small enough with respect to $n$ we obtain that
$$N_n(0,0)-\tilde N_n(0,0^+)\overset{n\rightarrow\infty}{\longrightarrow}0,$$
and the Proposition follows by using the next lemma. 

\begin{lemma}\label{lemmalimitN}
The following convergence holds:
$$\lim_{n\rightarrow\infty}\tilde N_n(0,0^+)=(0,\frac{1-i}{\sqrt{2}},\frac {-1-i}{\sqrt{2}}).$$
\end{lemma}
\begin{proof}
We recall the results at hand of $\tilde N_n(0,0^\pm)$: in view Lemma 4.4 and Lemma 4.7 of \cite{BV5} there exist rotations $\tilde\Theta_{n,k}$ such that
$$T_n(0,k^+)=\tilde\Theta_{n,k}(A^+_{|\alpha_{n,k}|}),\quad T_n(0,k^-)=\tilde\Theta_{n,k}(-A^-_{|\alpha_{n,k}|}),$$
$$\tilde N_n(0,k^\pm)=e^{i\sum_{j\neq k}|\alpha_{n,k}|^2\log|k-j|}e^{-i\arg(\alpha_{n,k})}\tilde\Theta_{n,k}(B^\pm_{|\alpha_{n,k}|}),$$
and 
$$T_n(0,0^\pm)=(\sin\frac{\theta_{n}}2,\pm\cos\frac{\theta_{n}}2,0),$$
to fit the directions of $\partial_x\chi_n(0,0^\pm)$ in \eqref{tang}.  
Therefore with respect to the notations in Lemma 4.4 and Lemma 4.7 of \cite{BV5}, the rotation $\tilde\Theta_{n,0}$ is determined by its values
$$\tilde\Theta_{n,0}(A^\pm_{c_n})=(\sin\frac{\theta_{n}}2,\pm\cos\frac{\theta_{n}}2,0).$$
In particular 
$$ \tilde\Theta_{n,0}(\frac{A^+_{c_n}\wedge A^-_{c_n}}{|A^+_{c_n}\wedge A^-_{c_n}|})=(0,0,1).$$
It follows that $\tilde N_n(0,0^\pm)$ is determined in terms of $\alpha_n, A^\pm_{c_n}$ and $B^\pm_{c_n}$. More precisely, we decompose
\begin{equation}\label{decB}
\left\{\begin{array}{c}\Re(B^+_{c_n})=a_{n,1} A^+_{c_n} +a_{n,2} A^-_{c_n} + a_{n,3} \frac{A^+_{c_n}\wedge A^-_{c_n}}{|A^+_{c_n}\wedge A^-_{c_n}|},\\\Im(B^+_{c_n})=b_{n,1} A^+_{c_n} +b_{n,2} A^-_{c_n} + b_{n,3} \frac{A^+_{c_n}\wedge A^-_{c_n}}{|A^+_{c_n}\wedge A^-_{c_n}|}.\end{array}\right.
\end{equation}
We have then
\begin{equation}\label{limNexpl}
\tilde N_n(0,0^+)=e^{ic_n^22\log(\left \lfloor{n^\nu}\right \rfloor!)}e^{-i\arg(\alpha_{n,0})}(\tilde\Theta_{n,0}(\Re(B^+_{c_n}))+i\tilde\Theta_{n,0}(\Im(B^+_{c_n})))\end{equation}
$$=e^{ic_n^22\log(\left \lfloor{n^\nu}\right \rfloor!)}e^{-i\arg(\alpha_{n,0})}$$
$$\times\left(((a_{n,1}+a_{n,2})\sin\frac{\theta_{n}}2,(a_{n,1}-a_{n,2})\cos\frac{\theta_{n}}2,a_{n,3})+i((b_{n,1}+b_{n,2})\sin\frac{\theta_{n}}2,(b_{n,1}-b_{n,2})\cos\frac{\theta_{n}}2,b_{n,3})\right).$$
Now to conclude we have to check if $\tilde N_n(0,0^+)$ has a limit. On one hand $\theta_n$ converges to $\pi$, and on the other hand we have $\arg(\alpha_{n,0})=0$. Thus for having a limit for $\tilde N_n(0,0^+)$ it is enough to show that  $a_{n,1}+a_{n,2},(a_{n,1}- a_{n,2})\cos\frac{\theta_{n}}2, a_{n,3}$ are convergent, and similarly for the $b-$coefficients.

We recall that from Theorem 1 {\em{(iii)-(iv)}} in \cite{GRV} we know that the unitary vectors $\Re (B^\pm_{c_n}), \Im (B^\pm_{c_n}), A^\pm_{c_n}$ satisfy:
\begin{equation}\label{AB}
B^\pm_{c_n}\perp A^\pm_{c_n}, \quad (A^+_{c_n,1}, A^+_{c_n,2}, A^+_{c_n,3})= (A^-_{c_n,1}, -A^-_{c_n,2}, -A^-_{c_n,3}),\quad A^+_{c_n,1}=e^{-\frac \pi 2 c_n^2},
\end{equation}
and in particular 
\begin{equation}\label{A}
\frac{A^+_{c_n}\wedge A^-_{c_n}}{|A^+_{c_n}\wedge A^-_{c_n}|}=\frac{(0,A^+_{c_n,3},-A^+_{c_n,2})}{\sqrt{1-e^{-\pi c_n^2}}},\quad \langle A^-_{c_n},A^+_{c_n}\rangle=2e^{-\pi c_{n}^2}-1.
\end{equation}
We also have from formulas (55), (47), (48), (69) and (56) in \cite{GRV}:
\begin{equation}\label{ABlim}
B^+_{c_n}\overset{n\rightarrow\infty}\approx (-c_n \sqrt{\frac \pi 2},-1, 0)+i(c_n \sqrt{\frac{\pi}{2}},0, -1),\quad A^+_{c_n}\overset{n\rightarrow\infty}\approx (e^{-\frac \pi 2 c_n^2}, -c_n\sqrt{\frac{\pi}{2}}, c_n\sqrt{\frac{\pi}{2}}).
\end{equation}

First, by looking at the first two coordinates of the decomposition \eqref{decB} we get
\begin{equation}\label{a1+a2}
(\Re(B^+_{c_n}))_1=(a_{n,1}+a_{n,2}) e^{-\frac \pi 2 c_n^2},
\end{equation}
and
\begin{equation}\label{a3}
(\Re(B^+_{c_n}))_2=(a_{n,1}-a_{n,2})A^+_{c_n,2}+a_{n,3}\frac{A^+_{c_n,3}}{\sqrt{1-e^{-\pi c_n^2}}},
\end{equation}
and similarly for  $b_n$ using $\Im (B^+_{c_n})$.

Secondly, from the orthogonality relation in \eqref{AB} we have
$$
\left\{\begin{array}{c}0=\langle  \Re(B^+_{c_n}),A^+_{c_n}\rangle=a_{n,1}+a_{n,2}\langle A^-_{c_n},A^+_{c_n}\rangle=a_{n,1}+a_{n,2}(2e^{-\pi c_{n}^2}-1),\\0=\langle  \Im(B^+_{c_n}),A^+_{c_n}\rangle=b_{n,1}+b_{n,2}\langle A^-_{c_n},A^+_{c_n}\rangle=b_{n,1}+b_{n,2}(2e^{-\pi c_{n}^2}-1),\end{array}\right.
$$
thus by using also \eqref{a1+a2},
\begin{equation}\label{a1a2}
\left\{\begin{array}{c}a_{n,1}=a_{n,2}(1-2e^{-\pi c_{n}^2}),\quad a_{n,2}=\frac{\Re(B^+_{c_n})e^{\frac \pi 2 c_n^2}}{2(1-e^{-\pi c_{n}^2})},\\b_{n,1}=b_{n,2}(1-2e^{-\pi c_{n}^2}),\quad b_{n,2}=\frac{\Im(B^+_{c_n})e^{\frac \pi 2 c_n^2}}{2(1-e^{-\pi c_{n}^2})}.\end{array}\right.
\end{equation}
Finally, by using \eqref{ABlim} we obtain 
\begin{equation}\label{a1a2next}
a_{n,1}\overset{n\rightarrow\infty}\approx \frac{1}{c_n\,2\sqrt{2\pi}},\quad a_{n,2}\overset{n\rightarrow\infty}\approx -\frac{1}{c_n\,2\sqrt{2\pi}},\quad b_{n,1}\overset{n\rightarrow\infty}\approx -\frac{1}{c_n\,2\sqrt{2\pi}},\quad b_{n,2}\overset{n\rightarrow\infty}\approx \frac{1}{c_n\,2\sqrt{2\pi}}.
\end{equation}
Then, in view of \eqref{ABlim} and \eqref{a3} we get 
\begin{equation}\label{a1a2a3bis}
a_{n,3}\overset{n\rightarrow\infty}\longrightarrow -\frac 1{\sqrt{2}},\quad b_{n,3}\overset{n\rightarrow\infty}\longrightarrow -\frac 1{\sqrt{2}}.
\end{equation}

In particular 
$$
(a_{n,1}- a_{n,2})\cos\frac{\theta_{n}}2\overset{n\rightarrow\infty}\longrightarrow\frac1{\sqrt{2}},\quad (b_{n,1}- b_{n,2})\cos\frac{\theta_{n}}2\overset{n\rightarrow\infty}\longrightarrow-\frac1{\sqrt{2}},
$$
and from \eqref{a1+a2} and \eqref{ABlim} we have
$$
a_{n,1}+ a_{n,2}\overset{n\rightarrow\infty}\longrightarrow 0,\quad b_{n,1}+ b_{n,2}\overset{n\rightarrow\infty}\longrightarrow 0.
$$
Then, from \eqref{limNexpl}, the convergence and the explicit limit of $\tilde N_n(0,0^+)$ as $n$ goes to infinity follow.

\end{proof}

\end{proof}

\section{Proof of Theorem \ref{th} in the nontrivial torsion case}\label{secthelix}
First we shall recall the construction of the solutions to the  binormal flow which at an initial time are given by non-planar polygonal lines.
In \cite{BV5} we showed that given a sequence of the complex numbers $\{\alpha_n\}_{n\in\mathbb Z}$, with at least two non-trivial values, starting from the Schr\"odinger solution in \eqref{ansatz} we obtain the BF evolution of a polygonal line $\chi(0)$ fully characterized modulo translation and rotation by the following description of its curvature and torsion angles. 

We denote $n_0,n_1\in\mathbb Z$ two consecutive locations where the sequence is non-trivial: $\alpha_{n_0}\neq 0, \alpha_{n_1}\neq 0, \alpha_{k}=0,\,\forall \, n_0<k<n_1$. 
We consider the ordered set of integers $\mathcal L=\{n_k\}_{k\in L}$ where the sequence does not vanish, containing in particular $n_0$ and $n_1$, so $L$ stand for a finite or infinite set of consecutive integers including $0$ and $1$.  Then $\mathcal L$ represents the locations of the corners of the polygonal line $\chi(0)$. At a corner located at $n_k\in \mathcal L$, the curvature angle is the one of the selfsimilar BF solution $\chi_{|\alpha_{n_k}|}$, that is $\theta_k$ given by formula \eqref{angle}:
$$\sin\frac{\theta_k} 2=e^{-\frac\pi 2|\alpha_{n_k}|^2}.$$
For $k,k+1\in L, k\geq 0$ the torsion is determined by the identities
\begin{equation}\label{argcoef}
\left\{\begin{array}{c}
\frac{T_{k-1}\wedge T_k}{|T_{k-1}\wedge T_k|}.\frac{T_k\wedge T_{k+1}}{|T_k\wedge T_{k+1}|}=-\cos(f_{|\alpha_{n_k}|}-f_{|\alpha_{n_{k+1}}|}+Arg(\alpha_{n_k})-Arg(\alpha_{n_{k+1}})) ,\\ 
\frac{T_{k-1}\wedge T_k}{|T_{k-1}\wedge T_k|}\wedge\frac{T_k\wedge T_{k+1}}{|T_k\wedge T_{k+1}|}=-sgn(\sin(f_{|\alpha_{n_k}|}-f_{|\alpha_{n_{k+1}}|}+Arg(\alpha_{n_k})-Arg(\alpha_{n_{k+1}})))T_k,
\end{array}\right.
\end{equation}
where we denoted $T_k=\partial_x\chi(0,x)$ for $x\in (n_k,n_{k+1})$, and $f$ is a function described in \cite{BV5} that we will not explicit here as we will work with sequences $\{\alpha_n\}_{n\in\mathbb Z}$ of same modulus.For negative subindices $k<0$ the torsion is defined in a similar way.\par
In particular, by considering the sequence of complex numbers in \eqref{defalpha} we obtain indeed the BF evolution of the helicoidal polygonal lines $\chi_n(0)$ in the statement of Theorem \ref{th}. \\

The proof of Theorem \ref{th} in the nontrivial torsion case goes the same as for the planar case treated in the previous subsections, except we get a modification of the term yielding Riemann's function in Lemma \ref{lemmaR}. More precisely, we obtain   
\begin{equation}\label{cvRhelix}
n(\chi_{n}(t,0)-\chi_{n}(0,0))- \frac{\Gamma}{2\sqrt{\pi}}\Im\,\bigl((0,\frac{1-i}{\sqrt{2}},\frac {-1-i}{\sqrt{2}})\,\int_{0}^{t}\sum_{|j|\leq n^\nu }e^{-ij\omega_0}\frac{e^{-i\frac{j^2}{4\tau}}}{\sqrt{\tau}}\,d\tau\bigr)\overset{n\rightarrow\infty}\longrightarrow 0,
\end{equation}
uniformly on $(0,T)$. We shall treat the integral as in the proof of Lemma \ref{lemmaR}. First the summation can be extended to the whole set of integers as by integrating by parts
$$|\int_{0}^{t}\sum_{|j|> n^\nu }e^{-ij\omega_0}\frac{e^{-i\frac{j^2}{4\tau}}}{\sqrt{\tau}}\,d\tau|\leq \frac {C}{n^\nu}.$$
Thus we have to analyze
\begin{equation}\label{int}
I(t)=\int_{0}^{t}\sum_{j\in\mathbb Z}e^{-ij\omega_0}e^{-i\frac{j^2}{4\tau}}\frac{d\tau}{\sqrt{\tau}}=\int_{0}^{t}e^{i\tau\omega_0^2}\sum_{j\in\mathbb Z}e^{-i\frac{(j+2\tau\omega_0)^2}{4\tau}}\frac{d\tau}{\sqrt{\tau}}.
\end{equation}
Now we use again Poisson's summation formula $\sum_{j\in\mathbb Z}f(j)=\sum_{j\in\mathbb Z}\hat f(2\pi j)$ to get
$$
\sum_{j\in\mathbb Z}e^{-i2\pi  jr}e^{i4\pi^2 tj^2}=\sum_{j\in\mathbb Z}\int e^{-i 2\pi x(j+r)+i4\pi^2 tx^2}dx=\frac 1{2\pi\sqrt{ t}}\sum_{j\in\mathbb Z}\int e^{-ix\frac {j+r}{\sqrt{t}}+ix^2}dx
$$
$$=\frac 1{2\pi\sqrt{ t}}\sum_{j\in\mathbb Z}\widehat{e^{i\cdot^2}}(\frac {j-r}{\sqrt{t}})=\frac {e^{i\frac\pi 4}}{2\sqrt\pi \sqrt{t}}\sum_{j\in\mathbb Z}e^{- i\frac{(j-r)^2}{4t}}.$$
By using this formula with $r=2\tau\omega_0$ we obtain
\begin{equation}\label{int2}
I(t)=2\sqrt{\pi}e^{-i\frac\pi 4}\int_{0}^{t}e^{i\tau\omega_0^2}\sum_{j\in\mathbb Z}e^{-i4\pi \tau \omega_0j}e^{i4\pi^2 \tau j^2}\,d\tau=2\sqrt{\pi}e^{-i\frac\pi 4}\int_{0}^{t}\sum_{j\in\mathbb Z}e^{i\pi^2 \tau (2j-\frac{\omega_0}{\pi})^2}\,d\tau.
\end{equation}
We recall now that $\omega_0$ is of rational type, $\omega_0=\frac ab\pi$, with $a,b$ with no common divisors, 
thus
$$I(t)=2\sqrt{\pi}e^{-i\frac\pi 4}\int_{0}^{t}\sum_{j\in\mathbb Z}e^{i\pi^2 \tau (2j-\frac ab)^2}\,d\tau=\frac{2b^2}{\pi \sqrt{\pi}}e^{-i\frac\pi 4}\sum_{j\in\mathbb Z}\frac{e^{i\pi^2 \frac t{b^2} (2bj-a)^2}-1}{i(2bj-a)^2}.$$
In view of \eqref{cvRhelix}-\eqref{int} we obtain   
\begin{equation}\label{cvRhelixbis}
n(\chi_{n}(t,0)-\chi_{n}(0,0))- \Im\,\bigl((0,-i,-1)\,\Gamma\sum_{j\in\mathbb Z}\frac{e^{i4\pi^2 t(j-\frac a{2b})^2}-1}{i4\pi^2(j-\frac a{2b})^2}\bigr)\overset{n\rightarrow\infty}\longrightarrow 0.
\end{equation}
The nontrivial torsion case of Theorem \ref{th} with 
\begin{equation}
\label{Rtilde} 
\frak{\tilde R}(t)=-\Gamma\sum_{j\in\mathbb Z}\frac{e^{i4\pi^2 t(j-\frac a{2b})^2}-1}{i4\pi^2(j-\frac a{2b})^2},
\end{equation}
follows from the following Propositions \ref{propspectrum}-\ref{propfractal}.

\begin{prop}\label{propspectrum} Let $n\in\mathbb N,m\in\mathbb N^*$. 
The spectrum of singularities of the function 
\begin{equation}\label{varR1}
\frak R_{n,m}(t)=\sum_{j\in\mathbb Z}\frac{e^{i2\pi t(mj-n)^2}-1}{(mj-n)^2},
\end{equation}
enjoys the same property \eqref{spectrum} of Riemann's function:
$$
d_{\frak R_{n,m}}(\beta)=4\beta-2, \quad \forall \beta \in [\frac 12,\frac 34].
$$
\end{prop}

\begin{proof}We first notice that Chamizo and Ubis prove in Theorem 2.3 of \cite{ChUb} that the above sum with denominator $j^2$ instead of $(mj-n)^2$ has the same spectrum as Riemann's function. The proof we give below  follows very closely their method and the one by Oskholkov and Chakhkiev \cite{OC}. \par
\medskip
We start with finding asymptotics of $\frak R_{n,m}$ near rational points. 
\begin{lemma} Let $p\in\mathbb N, q\in\mathbb N^*$ with no common divisors. There exist a positive constants $C$ such that we have the following estimate \footnote{$\sqrt{-1}=e^{i\frac\pi 2}$.}:
\begin{equation}\label{estdiff}
|\frak R_{n,m}(\frac pq+h)-\frak R_{n,m}(\frac pq)-\frac{z_0}{m}\frac {\tau_0}{q}\sqrt{h}|\leq 
C \min \bigl( |h|\sqrt{q},\, |hq|^{3/2}\bigr),
\end{equation}
where $z_0\in\mathbb C^*$ and
$$\tau_0=\sum_{r=0}^{q-1}e^{i2\pi\frac{p(m r-n)^2}{q}}.$$
For $q$ odd $\tau_0=\sqrt{q}$.
\end{lemma}

\begin{proof}
We decompose $j=ql+r$, with $l\in\mathbb Z$ and $0\leq r<q$ and write
$$\frak R_{n,m}(\frac pq)=\sum_{j\in\mathbb Z}\frac{e^{i2\pi \frac pq (mj-n)^2}-1}{(mj-n)^2}=\sum_{r=0}^{q-1}\sum_{l\in\mathbb Z}\frac{e^{i2\pi\frac pq (mr-n)^2}-1}{(m(ql+r)-n)^2}.$$
Then, assuming for simplicity that $h>0$,
$$\frak R_{n,m}(\frac pq+h)-\frak R_{n,m}(\frac pq)=\sum_{r=0}^{q-1}e^{i2\pi\frac pq (mr-n)^2}\sum_{l\in\mathbb Z}\frac{e^{i2\pi h (m(ql+r)-n)^2}-1}{(m(ql+r)-n)^2}.$$
We use Poisson's summation formula $\sum_{l\in\mathbb Z}f(l)=\sum_{l\in\mathbb Z}\hat f(2\pi l)$ in the second summation:
$$\frak R_{n,m}(\frac pq+h)-\frak R_{n,m}(\frac pq)=\sum_{r=0}^{q-1}e^{i2\pi\frac pq (mr-n)^2}\sum_{l\in\mathbb Z}\int \frac{e^{i2\pi h (m(qx+r)-n)^2}-1}{(m(qx+r)-n)^2}e^{-i2\pi l x}dx.$$
By changing variable $y=m(qx+r)-n$, and $s=\sqrt{h}y$, we have
$$\frak R_{n,m}(\frac pq+h)-\frak R_{n,m}(\frac pq)=\frac 1{dq}\sum_{r=0}^{q-1}e^{i2\pi\frac pq (mr-n)^2}\sum_{l\in\mathbb Z}\int \frac{e^{i2\pi h y^2}-1}{y^2}e^{-i2\pi l \frac{y-(mr-n)}{dq}}dy$$
$$=\frac 1{dq}\sum_{r=0}^{q-1}e^{i2\pi\frac pq (mr-n)^2}\sum_{l\in\mathbb Z}e^{i2\pi l \frac{(mr-n)}{dq}}\int \frac{e^{i2\pi h y^2}-1}{y^2}e^{-i2\pi l \frac{y}{mq}}dy$$
$$=\frac {\sqrt{h}}{mq}\sum_{r=0}^{q-1}e^{i2\pi\frac pq (mr-n)^2}\sum_{l\in\mathbb Z}e^{i2\pi l \frac{(mr-n)}{mq}}\int \frac{e^{i2\pi s^2}-1}{s^2}e^{-i2\pi l \frac{s}{\sqrt{h}mq}}ds$$
$$=\frac {\sqrt{h}}{mq}\sum_{r=0}^{q-1}e^{i2\pi\frac pq (mr-n)^2}\sum_{l\in\mathbb Z}e^{i2\pi l \frac{(mr-n)}{mq}}J(\frac{2\pi l }{\sqrt{h}mq}),$$
where
\begin{equation}\label{Jdef}
J(x)=\int \frac{e^{i2\pi s^2}-1}{s^2}e^{-isx}ds.
\end{equation}
We note that $J(0)\neq0$ is well-defined and  by integrating by parts twice we get for $|x|>1$, see \cite{OC} 
\begin{equation}\label{Jest}
|J(x)|\leq \frac{C}{x^2}.
\end{equation}
Thus 
$$\frak R_{n,m}(\frac pq+h)-\frak R_{n,m}(\frac pq)=\frac {\sqrt{h}}{mq}\sum_{l\in\mathbb Z}J(\frac{2\pi l }{\sqrt{h}mq})\sum_{r=0}^{q-1}e^{i2\pi\frac pq (mr-n)^2}e^{i2\pi l \frac{(mr-n)}{mq}}$$
$$=\frac {\sqrt{h}}{mq}\sum_{l\in\mathbb Z}J(\frac{2\pi l }{\sqrt{h}mq})\tau_l,$$
with
$$\tau_l=e^{i2\pi l\frac n{mq}}\sum_{r=0}^{q-1}e^{i2\pi\frac{p(mr-n)^2+2lr}{q}}.$$
Now we recall that the classical bounds on Gauss sums yields $|\tau_l|\leq \sqrt{2q}$, with $|\tau_l|=\sqrt{q}$ for $q$ qdd. Then by using the estimate \eqref{Jest} on $J(\frac{2\pi l }{\sqrt{h}mq})$ for all $l$ that gets summability in $l$, we obtain the upper-bound $|hq|^{3/2}$ in \eqref{estdiff} with $z_0=J(0)$. Moreover, for $\sqrt{t}q>$ we upper-bound $J(\frac{2\pi l }{\sqrt{h}mq})$ by a constant for $|l|\leq \sqrt{h} q$ and the estimate \eqref{Jest} for the remaining $l$'s. This yields the upper-bound $|h|\sqrt{q}$ in \eqref{estdiff} with $z_0=J(0)$, so the proof of the Lemma is complete.
\end{proof}

Based on these estimates around rational points one can follow the proof of Theorem 2.3 in \cite{ChUb} to get that the spectrum of singularities of $\frak R_{n,m}$ is given by \eqref{spectrum}. For the sake of completeness we give here shortly the argument, based on approximations by continued fractions (a.b.c.f.). 
Consider for $2\leq r\leq\infty$
$$A_r=\{x\in[0,1]\setminus \mathbb Q \mbox{ such that }   |x-\frac{p_k}{q_k}|=\frac1{q_k^{r_k}} \mbox{ with }\frac{p_k}{q_k}\mbox{ the a.b.c.f of }x,\, \lim\sup r_k= r\},$$
and $A_r^*$ the same set with the extra-condition that for a subsequence $q_{k_m}$ are odd numbers and $r_{k_m}$ converges to $r$. Concerning the Hausdorff dimensions of these sets one has from Jarn\'{i}k's theorems (\cite{Fa}):
\begin{equation}\label{dimH}
\dim_{\mathcal H} A_r=\dim_{\mathcal H} A_r^*=\dim_{\mathcal H}\cup_{s\geq r}A_s=\frac 2r.
\end{equation}
For $x\in A_r$ and a small $h\neq 0$ there exists $k$ such that
$$|x-\frac{p_k}{q_k}|=\frac1{q_k^{r_k}}\leq |h|<\frac1{q_{k-1}^{r_{k-1}}}=|x-\frac{p_{k-1}}{q_{k-1}}|.$$
As for a.b.c.f. $|x-\frac{p_k}{q_k}|\leq \frac 1{q_kq_{k+1}}$ we obtain
$$|h|^{-\frac 1{r_k}}\leq q_k<|h|^{-1+\frac 1{r_{k-1}}}.$$
Then, combining with \eqref{estdiff}, we have 
\begin{equation}\label{estdiffter}
|\frak R_{n,m}(x+h)-\frak R_{n,m}(x)|\leq|\frak R_{n,m}(x+h)-\frak R_{n,m}(\frac{p_k}{q_k})-\frac{z_0}{m}\frac {\tau_0}{q_k}\sqrt{x+h-\frac{p_k}{q_k}}|
\end{equation}
$$+|\frak R_{n,m}(x)-\frak R_{n,m}(\frac{p_k}{q_k})-\frac{z_0}{d}\frac {\tau_0}{q_k}\sqrt{x-\frac{p_k}{q_k}}|+|\frac{z_0}{d}\frac {\tau_0}{q_k}\sqrt{x+h-\frac{p_k}{q_k}}-\frac{z_0}{m}\frac {\tau_0}{q_k}\sqrt{x-\frac{p_k}{q_k}}|$$
$$\leq C h\sqrt{q_k}+C\frac{\sqrt{h}}{\sqrt{q_k}}\leq Ch^{\frac 12+\frac 1{2r_{k-1}}}+Ch^{\frac 12+\frac 1{2r_{k}}}.$$
Finally we note that for $q_k$ odd, by taking in \eqref{estdiff} the value $h_k=x-\frac {p_k}{q_k}$ so that $|h_k|=\frac1{ q_k^{r_k}}$, we have, as $r_k\geq 2$,
$$|\frak R_{n,m}(x)-\frak R_{n,m}(x-h_k)|\geq Ch_k^{\frac 12+\frac 1{2r_{k}}}-Ch_k^{\frac 32-\frac 3{2r_{k}}}\geq Ch_k^{\frac 12+\frac 1{2r_{k}}}.$$
Then at $x\in A_r^*$ the function $\frak R_{n,m}$ has local H\"older exponent $\frac 12+\frac 1{2r}$ and by \eqref{dimH} we get
$$d_{\frak R_{n,m}}(\frac 12+\frac 1{2r})\geq \dim_{\mathcal H} A_r^*=\frac 2r.$$
We obtain the converse inequality by noting that if the function $\frak R_{n,m}$ has local H\"older exponent $\frac 12+\frac 1{2r}$ at some $x\in [0,1]$, then from \eqref{estdiffter} the point $x$ is either in $\mathbb Q$ or in $ \cup_{s\geq r}A_s$, of which Hausdorff dimension is given in \eqref{dimH}. By varying $r\in [2,\infty]$ we recover the spectrum of $\frak R_{n,m}$ to be the one given in \eqref{spectrum}. \\

\end{proof}

\begin{prop}\label{propfractal} Let $n\in\mathbb N,m\in\mathbb N^*$. 
The function 
\begin{equation}\label{varR}
\frak R_{n,m}(t)=\sum_{j\in\mathbb Z}\frac{e^{i2\pi t(mj-n)^2}-1}{(mj-n)^2},
\end{equation} 
satisfies the multifractal formalism formula \eqref{FP}. 
\end{prop}

\begin{proof}In view of Proposition \ref{propspectrum} and of the multifractal formalism formula \eqref{FP} we have to compute
$$\eta_{\frak R_{n,m}}(p)=\sup\{s,\,\frak R_{n,m}\in B_p^{\frac sp,\infty}\}.$$
Recall that Riemann's function $\varphi_D$ is defined in \eqref{Duis}. Jaffard proved in \cite{Ja} the multifractal formalism formula \eqref{FP} for $\varphi_D$  using sharp bounds of  the $L^p$ norms of the partial sums of $\varphi_D'$ from \cite{Za}. In the following we shall obtain similar bounds for $\frak R_{n,m}'$. \\

Up to rescaling the variable $t$ by $\frac 1m$ we are thus interested in $L^p$ norms of dyadic partial sums of 
$$\sum_{j\in\mathbb Z} e^{i2\pi (tj^2-t\frac{2n}{m}j)}.$$

We shall need the following general lemma on partial sums of exponential sums. It is a classical type of result, initially motivated by the study of Vinogradov's mean value conjecture. The short proof we give is based on the explicit expression at rational times of the fundamental solution of the linear Schr\"odinger equation with periodic boundary conditions given in \eqref{Dirac3}. As it is well known this fundamental solution is intimately linked to the Talbot effect.

\begin{lemma}\label{estexpsums}
Let $N\in\mathbb N$ and $\sigma_N(x)=\sigma(\frac xN)$ where $\sigma$ is a smooth real positive function with compact support.  Let $a,b,q\in\mathbb N$, $1\leq a<q\leq N$, $(a,q)=1$ and $0\leq b<q$.\\
i) For $q$ odd there exists $\delta>0$, depending just on $\sigma$, such that 
\begin{equation}\label{equiv}
C\frac{N}{\sqrt{q}}\leq |\sum_{j\in\mathbb Z}\sigma_N(j)\,e^{i2\pi(tj^2-xj)}|\leq \tilde C\frac{N}{\sqrt{q}},
\end{equation}
if
$$|t-\frac {a}q|\leq \frac \delta {N^2}, \quad |x-\frac bq|\leq \frac \delta {N}.$$
ii) We have the upper-bound
\begin{equation}\label{upperbd}
|\sum_{j\in\mathbb Z}\sigma_N(j)\,e^{i2\pi(tj^2-xj)}|\leq C\frac{N}{\sqrt{q}(1+N\sqrt{|t-\frac aq|})},
\end{equation}
if
$$|t-\frac {a}q|< \frac 1{qN}.$$

\end{lemma}

\begin{remark}
The bounds in \eqref{equiv} can be found in the literature for $q$ varying up to $\sqrt{N}$, with $\delta=1$ (see Lemma 1 in page 56 of \cite{Vi}, or see for instance (2.46)-(2.47) in \cite{Bo}). A different proof of the upper-bound \eqref{upperbd} for $q<N$ can be found in Lemma 3.18 in \cite{Bo}.
\end{remark}

\begin{proof}We denote $\tilde t=\frac t{2\pi}$.
For $\tilde t=\frac 1{2\pi}\frac aq+h$, we see the (conjugated) sum as a linear Schr\"odinger solution:
$$S_N(t,x):=\sum_{j\in\mathbb Z}\sigma_N(j)\,e^{-i4\pi^2\tilde tj^2+i2\pi x}=(e^{ih\Delta}e^{i\frac aq\Delta}(\sum_{j\in\mathbb Z}\sigma_N(j)\,e^{i2\pi\cdot j}))(x)$$
$$=(e^{ih\Delta}e^{i\frac aq\Delta}(\mathcal F^{-1}(\sigma_N)\star \mathcal F^{-1}(\sum_{j\in\mathbb Z} e^{i2\pi\cdot j})))(x)=(e^{ih\Delta}e^{i\frac aq\Delta}(\mathcal F^{-1}(\sigma_N)\star (\sum_{j\in\mathbb Z} e^{i2\pi\cdot j})))(x)$$
$$=(e^{ih\Delta}(\mathcal F^{-1}(\sigma_N)\star e^{i\frac 1{2\pi}\frac aq\Delta}(\sum_{j\in\mathbb Z} e^{i2\pi\cdot j})))(x).$$
Now we use the fact that (for instance choosing $M=2\pi$ in formulas (37) combined with (42) from \cite{DHV}):
$$e^{i\frac 1{2\pi}\frac aq\Delta}(\sum_{j\in\mathbb Z} e^{i2\pi\cdot j})(x)=\sum_{j\in\mathbb Z} \tau_j\delta(x-\frac jq),$$
with the coefficients $\tau_j$ given in terms of Gauss sums. In particular $|\tau_j|=\frac 1{\sqrt{q}}$ if $q$ is odd and $|\tau_j|\leq\frac {\sqrt2}{\sqrt{q}}$.
Then we have
$$S_N(t,x)=(e^{ih\Delta}(\sum_{j\in\mathbb Z} \tau_j \mathcal F^{-1}(\sigma_N)(\cdot -\frac jq))(x)=\sum_{j\in\mathbb Z} \tau_j(e^{ih\Delta}( \mathcal F^{-1}(\sigma_N)))(x-\frac jq)$$
$$=\sum_{j\in\mathbb Z} \tau_j \int e^{i(x-\frac jq)\xi} e^{-ih\xi^2}\sigma_N(\xi)d\xi=N\sum_{j\in\mathbb Z} \tau_j\int e^{i(x-\frac jq)N\xi} e^{-ihN^2\xi^2}\sigma(\xi)d\xi$$
$$=N\sum_{j\in\mathbb Z} \tau_j \int e^{i(x-\frac bq+\frac {b-j}q)N\xi} e^{-ihN^2\xi^2}\sigma(\xi)d\xi$$
$$=Nc_{b}\int e^{i(x-\frac bq)N\xi} e^{-ihN^2\xi^2}\sigma(\xi)d\xi$$
$$+N\sum_{j\neq b} \tau_j\int  e^{i(x-\frac bq+\frac {b-j}q)N\xi} e^{-ihN^2\xi^2}\sigma(\xi)d\xi=:I_1+I_2.$$
The first term $I_1$ gives the right growth and is not zero if  $q$ is odd as $|hN^2|<\delta$ and $|x-\frac bq|<\frac\delta N$. The last sum $I_2$ can be upper-bounded by two integrations by parts from the $\xi-$linear phase, that insures summation in $j$ and the upper-bound $C\frac{N}{\sqrt{q}}(\frac q{N})^2$. Thus  \eqref{equiv} follows. 

For getting the upper-bound \eqref{upperbd} we write:
$$S_N(t,x)=N\sum_{j=qx}c_{j} \int e^{i(x-\frac jq)N\xi} e^{-ihN^2\xi^2}\sigma(\xi)d\xi$$
$$+N\sum_{j\neq qx} \tau_j \int e^{i(x-\frac jq)N\xi} e^{-ihN^2\xi^2}\sigma(\xi)d\xi=:\tilde I_1+\tilde I_2.$$
In $\tilde I_2$ we use again integrations by parts yield the upper-bound $C\frac{N}{\sqrt{q}}$ for $\tilde I_2$. 
For $|t-\frac aq|\leq \frac \delta {N^2}$ the upper-bound \eqref{upperbd} follows by using the fact that the integrals in $\tilde I_1$ is bounded by a constant depending only on $\sigma$ . For $\frac \delta{N^2}\leq |t-\frac aq|\leq \frac 1{qN}$ we get \eqref{upperbd} by using the dispersion inequality on $\tilde I_1$ yielding the upper-bound $C\frac{N}{\sqrt{q}}\frac1{\sqrt{hN^2}}$. \end{proof}

For $N\in\mathbb N$ we consider the $L^p$ norm of the dyadic partial sum of $\frak R_{n,m}'$:
$$I_{N,p}:=\int _0^1 |\sum_{j\in\mathbb Z} \sigma_N(j)\,e^{i2\pi (tj^2-t\frac{2n}{m}j)}|^pdt,$$
with $\sigma$ a smooth real positive function valued $1$ on $1<|x|<2$ and vanishing on $|x|<\frac 12$ and $|x|>4$

First we get a lower bound by applying \eqref{equiv} with $x=t\frac{2n}{m}$, $a=m\tilde a, b=2n\tilde a$, with $q,N$ large with respect to $n,m$, with $(m\tilde a,q)=1$. Indeed then the condition on $x$ is satisfied as for these choices of parameters we have 
$$|t-\frac {m\tilde a}q|\leq \frac \delta{N^2} \Longrightarrow |x-\frac {2n\tilde a}q|=|t\frac{2n}{m}-\frac {2n\tilde a}q|\leq \frac{2n}{m}\frac \delta{N^2}+|\frac {m\tilde a}q\frac{2n}{m}-\frac {2n\tilde a}q|=\frac{2n}{m}\frac \delta{N^2}\leq \frac \delta N.$$
Thus we get the following lower bound by integrating on one region $|t-\frac {m\tilde a}q|\leq \frac \delta{N^2}$ for  instance for $q=2$ by applying \eqref{equiv}:
$$I_{N,p}\geq \int _{|t-\frac {m\tilde a}q|\leq \frac \delta{N^2}} |\sum_{j\in\mathbb Z} \sigma_N(j)\,e^{i2\pi (tj^2-\frac{2tn}{m}j)}|^pdt\geq C N^{p-2},$$
that will suit our purposes for $p>4$. For $p= 4$ we shall need to improve this lower bound by integrating on the union of all the disjoint regions $|t-\frac {m\tilde a}{q}|\leq \frac \delta{N^2}$ for all $q\leq N$. We shall use rough estimates that are enough for our porposes:
$$I_{N,p}\geq  C N^{2}\sum_{q\leq N} \frac{\sharp \{1\leq m\tilde a<q,(m\tilde a,q)=1\}}{q^2}
\geq C N^{2}\sum_{q\leq N,\, q \mbox{ prime }}\frac{\lfloor \frac q m\rfloor-1}{q^2}$$
$$\geq C(m) N^{2}\sum_{q\leq N,\, q \mbox{ prime }}\frac{1}{q}\geq C(m) N^{2}\sum_{j=1}^{\lfloor \log N\rfloor}\sum_{2^{j-1}\leq q<2^j,\, q \mbox{ prime }}\frac{1}{q}$$
$$\geq C(m) N^{2}\sum_{j=1}^{\lfloor \log N\rfloor}\frac{1}{2^j}\sharp\{q, 2^{j-1}\leq q<2^j,\, q \mbox{ prime}\}.$$
We use now the law of distribution of prime numbers, i.e. that the number of prime numbers less than a given number $x$ grows as $\frac {x}{\log x}$ to get
$$I_{N,p} \geq C(m)\, N^{2}\sum_{j=1}^{\lfloor \log N\rfloor}\frac{1}{j}\,\geq C(m) N^2\log (\log N).$$
For the remaining case $p<4$ we do the above calculation but just in the region $N/2<q\leq N$.
Thus we get the lower-bound:
\begin{equation}\label{Lpequiv}
I_{N,p}\geq C\left\{\begin{array}{c} N^{p-2}, \quad p>4,\\ N^2\log (\log N),\quad  p=4,\\ \frac{N^\frac p2}{\log  N}, \quad 0<p<4.\end{array}\right.\end{equation}

Finally we shall get upper-bounds almost similar to the above lower-bound for $I_{N,p}$. 
From Dirichlet principle we have the following covering of $[0,1]$: for all $t\in [0,1]$, there is $a\in\mathbb N,q\in\mathbb N^*, (a,q)=1, q\leq N$ such that
$$|t-\frac aq|<\frac 1{qN}.$$ 
Then from \eqref{upperbd},
$$I_{N,p}\leq C_p\int _0^1 |\sum_{j\in\mathbb Z} \sigma_N(j)\,e^{i2\pi (tj^2-t\frac{2n}{m}j)}|^pdt$$
$$\leq C_{m,p}\sum_{q=1}^{N}\sum_{a=0}^{q-1}\int_{|t-\frac aq|<\frac 1{qN}}|\sum_{j\in\mathbb Z}\sigma_N\,e^{i2\pi (tj^2-t\frac{2n}{m}j)}|^pdt$$
$$\leq C_{m,p}\sum_{q=1}^{N}q\left(\frac{N^{p-2}}{q^\frac p2}+\frac{1}{q^\frac p2}\int_{\frac 1{N^2}<|y|<\frac1{qN}}\frac 1{|y|^\frac p2}dy\right).$$
For the first term we have the following estimate
$$\sum_{q=1}^{N}q\frac{N^{p-2}}{q^\frac p2}\leq C\left\{\begin{array}{c} N^{p-2}, \quad p>4, \\ N^2\log N,\quad  p=4,\\N^\frac p2, \quad 0<p<4,\end{array}\right.$$
For the second term we use:
$$\int_{\frac 1{N^2}|y|<\frac1{qN}}\frac 1{|y|^\frac p2}dy \leq C\left\{\begin{array}{c} N^{p-2}, \quad p>2, \\ \log \frac Nq,\quad  p=2,\\\frac 1{(qN)^{1-\frac p2}}, \quad 0<p<2,\end{array}\right.$$
that implies
$$\sum_{q=1}^{N}\frac{q}{q^\frac p2}\int_{\frac 1{N^2}|y|<\frac1{qN}}\frac 1{|y|^\frac p2}dy \leq C\left\{\begin{array}{c} N^{p-2}, \quad p>4, \\ N^2\log N,\quad  p=4,\\ N^\frac p2,\quad 0<p<4,\end{array}\right.$$
and thus the same upper-bounds for $I_{N,p}$. 
Combining with \eqref{Lpequiv} we have obtained
\begin{equation}\label{boundsLp}
C\left\{\begin{array}{c} N^{p-2}, \quad p>4, \\ N^2\log (\log N),\quad  p=4,\\ \frac{N^\frac p2}{\log N},\quad 0<p<4,\end{array}\right.\leq \int_0^{1}|\sigma_N\sum_{j\in\mathbb Z} e^{i2\pi t(j^2-\frac {2n}mj)}|^pdt\leq C\left\{\begin{array}{c} N^{p-2}, \quad p>4,\\ N^2\log N,\quad  p=4,\\ N^\frac p2, \quad 0<p<4.\end{array}\right.
\end{equation}
Now we note that the function is $m-$periodic, with frequencies $j(j-\frac {2n}{m})$. 
As $D^k\leq |j(j-\frac{2n}m)|<D^{k+1}$ implies $C_{n,m,D}D^\frac k2<|j|<\tilde C_{n,m,D}D^\frac {k+1}2$ we can 
use \eqref{boundsLp}, where all the powers of $N$ in the equivalent are positive, to obtain for large $D-$adic blocs $\Delta_k$:
$$C\left\{\begin{array}{c} D^{\frac{k(p-2)}{2p}}, \quad p>4, \\ D^\frac k2\log^\frac 12 \log D^k,\quad  p=4,\\ \frac{D^\frac k4}{\log D^k},\quad 0<p<4,\end{array}\right.\leq\|\Delta_k\sum_{j\in\mathbb Z}e^{i2\pi t(j-\frac nm)^2}\|_{L^p}\leq C\left\{\begin{array}{c} D^{\frac{k(p-2)}{2p}}, \quad p>4,\\ D^\frac k2\log^\frac 12 D^k,\quad p=4,\\ D^\frac k4, \quad 0<p<4. \end{array}\right.$$
Therefore we get that $\frak R_{n,m}'$ belongs to $B_p^{-\frac 14,\infty}$ for $0<p<4$ and to $B_p^{-\frac 12+\frac 1p,\infty}$ for $p>4$, and these are optimal, yielding
$$\eta_{\frak R_{n,m}}(p)=\sup\{s,\,\frak R_{n,m}\in B_p^{\frac sp,\infty}\}=\left\{\begin{array}{c} 1+\frac p2, \quad p\geq4,\\\frac{3p}4, \quad 0<p<4. \end{array}\right.$$
Then the relation \eqref{FP} follows as for $\beta\in[\frac 12,\frac 34]$ we have 
$$\inf_p (\beta p-\eta_{\frak R_{n,m}}(p)+1)-d_{\frak R_{n,m}}(\beta)$$
$$=\inf(\{(\beta-\frac 12)(p-4), p\geq4\}\cup \{(p-4)(\beta-\frac{3p}4), 0<p<4\})=0. $$
\end{proof}

\begin{remark}
Riemann's function \eqref{1.0} is the integral of the Jacobi theta function of one variable $\vartheta(t)=\sum_{j\in\mathbb Z}e^{i\pi t j^2}$, and it follows from Jaffard's resulrs in \cite{Ja} that it is a multifractal function satisfying the multifractal formalism of Frish and Parisi \eqref{FP}. From Propositions \ref{propspectrum}-\ref{propfractal} it follows that these properties are also valid for the integrals of the Jacobi theta companions functions, for instance $\tilde{\tilde{\vartheta}}(t)=\sum_{j\in\mathbb Z}e^{i\pi t (j+\frac 12)^2}$.
\end{remark}

\bigskip

{\bf{Acknowledgements:}}  This research is partially supported by the Institut Universitaire
de France, by the French ANR project SingFlows, by ERCEA Advanced Grant 2014 669689
- HADE, by MEIC (Spain) projects Severo Ochoa SEV-2017-0718, and PGC2018-1228
094522-B-I00, and by Eusko Jaurlaritza project IT1247-19 and BERC program.

\end{document}